%% file: article_boundaries.tex
\documentclass{amsart}

\input{local_def.tex}

\begin{document}

\title[Bijections for planar maps with boundaries]{Bijections for planar maps with boundaries}

\author[O. Bernardi and \'E. Fusy]{Olivier Bernardi$^{*}$ \and \'{E}ric Fusy$^{\dagger}$}
\thanks{$^{*}$Department of Mathematics, Brandeis University, Waltham MA, USA,
bernardi@brandeis.edu.\\
\phantom{1}\ \ \ \!$^{\dagger}$LIX, \'Ecole Polytechnique, Palaiseau, France, fusy@lix.polytechnique.fr.}

\date{\today}


\begin{abstract}
We present bijections for planar maps with boundaries. In particular, we obtain bijections for triangulations and quadrangulations of the sphere with boundaries of prescribed lengths. For triangulations we recover the beautiful factorized formula obtained by Krikun using a (technically involved) generating function approach. The analogous formula for quadrangulations is new. 
We also obtain a far-reaching generalization for other face-degrees. In fact, all the known enumerative formulas for maps with boundaries are proved bijectively in the present article (and several new formulas are obtained).

Our method is to show that maps with boundaries can be endowed with certain ``canonical'' orientations, making them amenable to the master bijection approach we developed in previous articles. As an application of our enumerative formulas, we note that they provide an exact solution of the dimer model on rooted triangulations and quadrangulations. 
\end{abstract}

\maketitle

\section{Introduction}
In this article, we present bijections for planar maps with boundaries. Recall that a \emph{planar map} is a decomposition of the 2-dimensional sphere into vertices, edges, and faces, considered up to continuous deformation (see precise definitions in Section~\ref{sec:master-bij}). We deal exclusively with \emph{planar} maps in this article and call them simply \emph{maps} from now on. A \emph{map with boundaries} is a map with a set of distinguished faces called \emph{boundary faces} which are pairwise vertex-disjoint, and have simple-cycle contours (no pinch points). 
We call \emph{boundaries} the contours of the boundary faces.
We can think of the boundary faces as holes in the sphere, and maps with boundaries as a decomposition of a sphere with holes into vertices, edges and faces. A \emph{triangulation with boundaries} (resp. \emph{quadrangulation with boundaries}) is a map with boundaries such that every non-boundary face has degree 3 (resp. 4).

The main results obtained in this article are bijections for triangulations and quadrangulations with boundaries. The bijection establishes a correspondence between these maps and certain types of plane trees. This, in turns, easily yields factorized enumeration formulas with control on the number and lengths of the boundaries. In the case of triangulations, the enumerative formula had been established by Krikun~\cite{Kri} (by a technically involved ``guessing/checking'' generating function approach). The case of quadrangulations is new. We also present a far-reaching generalization for maps with other face-degrees.

The strategy we apply is to adapt to maps with boundaries the ``master bijection'' approach we developed in \cite{BeFu12,BeFu12b} for maps without boundaries. Roughly speaking, this strategy reduces the problem of finding bijections, to the problem of exhibiting canonical orientations characterizing these classes of maps. 

\fig{width=0.7\linewidth}{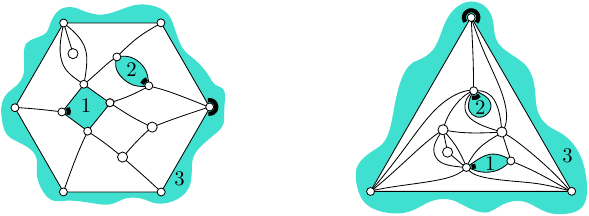}{Left: a quadrangulation in $\cQ[3;4,2,6]$. Right: a triangulation in $\cT[3;2,1,3]$.}

Let us now state the enumerative formulas derived from our bijections for triangulations and quadrangulations. We call a map with boundaries  \emph{multi-rooted} if the $r$ boundary faces are labeled with distinct numbers 
in $[r]=\{1,\ldots,r\}$, and each one has a marked corner; see Figure~\ref{fig:examples_maps}. For $m\geq 0$ and $a_1,\ldots,a_r$ positive integers, we denote $\cT(m;a_1,\ldots,a_r)$ (resp. $\cQ(m;a_1,\ldots,a_r)$) the set of multi-rooted triangulations (resp. quadrangulations) with $r$ boundary faces, and $m$ internal vertices (vertices not on the boundaries), such that the boundary labeled $i$ has length $a_i$ for all $i\in[r]$.
In 2007 Krikun proved the following result:
\begin{theo}[Krikun~\cite{Kri}]\label{theo:krikun_triang}
For $m\geq 0$ and $a_1,\ldots,a_r$ positive integers,
\begin{equation}\label{eq:krikun_triang}
|\cT[m;a_1,\ldots,a_r]|=\frac{4^k(e-2)!!}{m!(2b+k)!!}\prod_{i=1}^r a_i\binom{2a_i}{a_i},
\end{equation}
where $b:=\sum_{i=1}^r a_i$ is the total boundary length, $k:=r+m-2$, and $e=2b+3k$ is the number of edges (and the notation $n!!$ stands for $\prod_{i=0}^{\lfloor (n-1)/2\rfloor}(n-2i)$). 
\end{theo}
We obtain a bijective proof of this result, and also prove the following analogue:
\begin{theo}\label{theo:krikun_quad}
For $m\geq 0$ and $a_1,\ldots,a_r$ positive integers, 
\begin{equation}\label{eq:krikun_quad}
|\cQ[m;2a_1,\ldots,2a_r]|=\frac{3^k(e-1)!}{m!(3b+k)!}\prod_{i=1}^r2a_i\binom{3a_i}{a_i}, 
\end{equation}
where $b:=\sum_{i=1}^r a_i$ is the half-total boundary length, $k:=r+m-2$, and $e=3b+2k$ is the number of edges. 
\end{theo}

Equations~\eqref{eq:krikun_triang} and \eqref{eq:krikun_quad} are generalizations of classical formulas. Indeed, the doubly degenerate case $m=0$ and $r=1$ of~\eqref{eq:krikun_triang} gives the well-known \emph{Catalan formula} for the number of triangulations of a polygon without interior points $|\cT[0;a]|=\Cat(a-2)=\frac{(2a-4)!}{(a-1)!(a-2)!}$. Similarly, the case $m=0$ and $r=1$ of~\eqref{eq:krikun_quad} gives the \emph{2-Catalan formula} for the number of quadrangulations of a polygon without interior points $|\cQ[0;2a]|=\frac{(3a-3)!}{(a-1)!(2a-1)!}$. 
The case $r=1$, $a=1$ of~\eqref{eq:krikun_quad} is already non-trivial as it gives the well-known formula for the number of rooted quadrangulations with $m+2$ vertices (upon seeing the root-edge as blown into a boundary face of degree $2$):
$$|\cQ[m;2]|=\frac{2\cdot 3^m(2m)!}{m!(m+2)!}.$$
More generally, the case $r=1$ of~\eqref{eq:krikun_quad} yields
$$|\cQ[m;2a]|=\frac{3^{m-1}(3a+2m-3)!}{m!(3a+m-1)!}\frac{(3a)!}{(2a-1)!a!},$$ 
which is the formula given in~\cite[Eq.(2.12)]{bouttier2009distance} for the number of rooted quadrangulations with one simple boundary of length $2a$, and $m$ internal vertices. Similarly, the case $r=1$ of~\eqref{eq:krikun_triang} yields the formula for the number of rooted triangulations with one simple boundary of length $a$, but it seems that this formula was not known prior to \cite{Kri}.
Lastly, in Section~\ref{sec:counting} we use the special case of~\eqref{eq:krikun_triang} and \eqref{eq:krikun_quad} where all the boundaries have length 2 in order to solve the dimer model on triangulations and quadrangulations.

As a side remark, let us discuss the counterparts of 
\eqref{eq:krikun_triang} and~\eqref{eq:krikun_quad} when we remove the condition for the boundaries to be simple and pairwise disjoint. Let $\hcT[n;a_1,\cdots,a_r]$ (resp. $\hcQ[n;a_1,\ldots,a_r]$) be the set of maps with $n+r$ faces, $n$ faces of degree $3$ (resp. $4$) and
 $r$ distinguished faces labeled $1,\ldots,r$ of respective degrees $a_1,\ldots,a_r$, each having a marked corner. 
It is easy to deduce from Tutte's slicings formula~\cite{T62b} that 
$$
|\hcQ[n;2a_1,\ldots,2a_r]|=\frac{(e-1)!}{v!n!}3^n\prod_{i=1}^r2a_1\binom{2a_i-1}{a_i},
$$
where $v=n+2+\sum_{i=1}^r(a_i-1)$ is the total number of vertices, and $e=2n+\sum_{i=1}^r a_i$ is the total number of edges. 
However no factorized formula should exist for $|\hcT[n;a_1,\ldots,a_r]|$, since the formula for $r=1$ is already complicated~\cite{mathai1972enumeration}.


As mentioned above, we have also generalized our results to other face degrees. For these extensions, there is actually a necessary ``girth condition'' to take into account in order to obtain bijections. Precisely, we define a notion of \emph{internal girth} for plane maps with boundaries. The internal girth coincides with the girth\footnote{We recall that the girth of a graph $G$ is the length of a 
shortest cycle of edges in $G$.} when the map has at most one boundary (but can be larger than the girth in general). For any integer $d\geq 1$, we obtain a bijection for maps with boundaries having internal girth $d$, and non-boundary faces of degrees in  $\{d,d+1,d+2\}$ (with control on the number of faces of each degree). 
For $d=1$, the internal girth condition is void, and restricting the non-boundary faces to have degree $d+2=3$ gives our result for triangulations with boundaries.
For $d=2$, the internal girth condition is void for bipartite maps, and restricting the non-boundary faces to have degree $d+2=4$ gives our result for bipartite quadrangulations with boundaries. 
For the values of $d\geq 3$, the case of a single boundary with all the internal faces of degree $d$ corresponds to the results obtained in \cite{BeFu12} (bijections for $d$-angulations of girth $d\geq 3$ with at most one boundary). 
For $d=2$, the case of a single boundary with all the internal faces of degree $3$ gives a bijection for loopless triangulations (i.e. triangulations of girth at least 2) with a single boundary and we recover the counting formula of Mullin \cite{mullin1965counting}. Hence, our bijections cover the cases of triangulations with a single boundary with girth at least $d$, for $d\in\{1, 2,3\}$ (for girth 1 we give the first bijective proof, while for girth 2 the first bijective proof was given in \cite{PS03a} and for girth 3 it was given in~\cite{PS03b}, and generalized to $d$-angulations in~\cite{albenque2013generic}). Furthermore, in Theorem \ref{thm:krikun-analogues} we give generalizations of these results in the form of multivariate factorized counting formulas, analogous to Krikun formula \eqref{eq:krikun_triang}, for the  classes of triangulations of internal girth $d=2$ and $d=3$. 
Lastly,  we give  multivariate factorized counting formulas for the  classes of quadrangulations of internal girth $d=4$ thereby generalizing the formula of Brown~\cite{Br65} for simple quadrangulations with a single boundary. In fact, all the known counting formulas for maps with boundaries are proved bijectively in the present article.


This article is organized as follows. In Section~\ref{sec:master-bij} we set our definitions about maps, and adapt the \emph{master bijection} established in \cite{BeFu12} to maps with boundaries. In Section~\ref{sec:bij-quad}, we define canonical orientations for quadrangulations with boundaries, and obtain a bijection with a class of trees called \emph{mobiles} (the case where at least one boundary has size 2 is simpler, while the general case requires to first cut the map into two pieces). In Section~\ref{sec:bij-triang} we treat similarly the case of triangulations. In Section~\ref{sec:counting}, we count mobiles and obtain \eqref{eq:krikun_triang} and \eqref{eq:krikun_quad}. We also derive
from our formulas (both for coefficients and generating functions) exact solutions of the dimer model on rooted quadrangulations and triangulations. 
In Section~\ref{sec:girth} we unify and extend the results (orientations, bijections, and enumeration) to more general face-degree conditions. 
In Section~\ref{sec:proofs}, we prove the existence and uniqueness of the needed canonical orientations for maps with boundaries. Lastly, in Section~\ref{sec:extension}, we discuss additional results and perspectives.


\section{Maps and the master bijection}\label{sec:master-bij}
In this section we set our definitions about maps and orientations. We then recall the master bijection for maps established in \cite{BeFu12}, and adapt it to maps with boundaries.

\subsection{Maps and weighted biorientations}
A \emph{map} is a decomposition of the 2-dimensional sphere into \emph{vertices} (points), \emph{edges} (homeomorphic to open segments), and \emph{faces} (homeomorphic to open disks), considered up to continuous deformation. A map can equivalently be defined as a drawing (without edge crossings) of a connected graph in the sphere, considered up to continuous deformation. Each edge of a map is thought as made of two \emph{half-edges} that meet in its middle. A \emph{corner} is the region between two consecutive half-edges around a vertex. The \emph{degree} of a vertex or face $x$, denoted $\deg(x)$, is the number of incident corners. A \emph{rooted map} is a map with a marked corner $c_0$; the incident vertex $v_0$ is called the root vertex, and the half-edge (resp. edge) following $c_0$ in clockwise order around $v_0$ is called the \emph{root half-edge} (resp. \emph{root edge}).
A map is said to be \emph{bipartite} if the underlying graph is bipartite, which happens precisely when every face has even degree. 
A \emph{plane map} is a map with a face distinguished as its \emph{outer face}. We think about plane maps, as drawn in the plane, with the outer face being the infinite face.
The non-outer faces are called \emph{inner faces}; vertices and edges are called \emph{outer} or \emph{inner} depending on whether they are incident to the outer face or not; an half-edge is \emph{inner} if it belongs to an inner edge and \emph{outer} if it belongs to an outer edge. The degree of the outer face is called the \emph{outer degree}.

A \emph{biorientation} of a map $M$ is the assignment of a direction
to each half-edge of $M$, that is, each half-edge is either outgoing or ingoing at its incident vertex. 
For $i\in\{0,1,2\}$, an edge is called $i$-way if it has $i$ ingoing half-edges. 
An \emph{orientation} is a biorientation such that every edge is 1-way.
If $M$ is a plane map endowed with a biorientation, then a \emph{ccw cycle} (resp. \emph{cw-cycle}) of $M$ is a simple cycle $C$ 
of edges of $M$ such that each edge of $C$ is either 2-way or 1-way with the interior of $C$ on its left (resp. on its right). 
The biorientation is called \emph{minimal} if there is no ccw cycle, and 
\emph{almost-minimal} if the only ccw cycle is the outer face contour (in which case 
the outer face contour must be a simple cycle). 
For $u,v$ two vertices of $M$, $v$ is said to be \emph{accessible} from $u$ if there is a path $P=u_0,u_1,\ldots,u_k$ of vertices of $M$ such that $u_0=u$, $u_k=v$, and for $i\in[1..k-1]$, 
the edge $(u_i,u_{i+1})$ is either 1-way from $u_i$ to $u_{i+1}$ or 2-way.
The biorientation is said to be \emph{accessible} from $u$ if every vertex of $M$ is accessible from $u$. 
A \emph{weighted biorientation} of $M$ is a biorientation of $M$ where each half-edge is assigned a weight (in some additive group). A \emph{$\ZZ$-biorientation} is a weighted biorientation such that weights at ingoing half-edges are positive integers, while weights at outgoing half-edges are non-positive integers.

\subsection{Master bijection for $\ZZ$-bioriented maps}
We first define the families of bioriented maps involved in the master bijection. Let $d$ be a positive integer. 
We define $\cO_d$ as the set of plane maps of outer degree $d$ endowed with a $\ZZ$-biorientation which is minimal and accessible from every outer vertex, and such that every outer edge is either 2-way or is 1-way with an inner face on its right. 
We define $\cO_{-d}$ as the set of plane maps of outer degree $d$ endowed with a $\ZZ$-biorientation which is almost-minimal and accessible from every outer vertex, and such that outer edges are 1-way with weights $(0,1)$, and each inner half-edge incident to an outer vertex is outgoing. 

Next, we define the families of trees involved in the master bijection.
We call \emph{mobile} an unrooted plane tree with two kinds of vertices, black vertices and white vertices (vertices of the same color can be adjacent), where each corner at a black vertex possibly carries additional dangling half-edges called \emph{buds}; see Figure~\ref{fig:master_bij_weighted_biori} (right) for an example. The \emph{excess} of a mobile is 
defined as the number of half-edges incident to a white vertex, minus the number of buds. A \emph{weighted mobile} is a mobile where each half-edge, except for buds, is assigned a weight. A $\ZZ$-mobile is a weighted mobile such that weights of half-edges incident to white vertices are positive integers, while weights at half-edges incident to black vertices are non-positive integers. For $d\in\mathbb{Z}$, we denote by $\cB_d$ the set of $\ZZ$-mobiles of excess $d$. 

\fig{width=.7\linewidth}{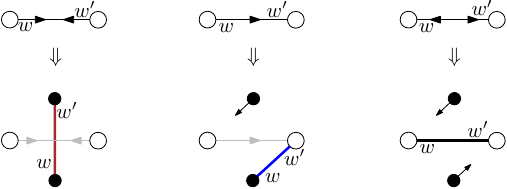}{The local rule performed at each edge (0-way, 1-way or 2-way) in the master bijection $\Phi$.}

\fig{width=\linewidth}{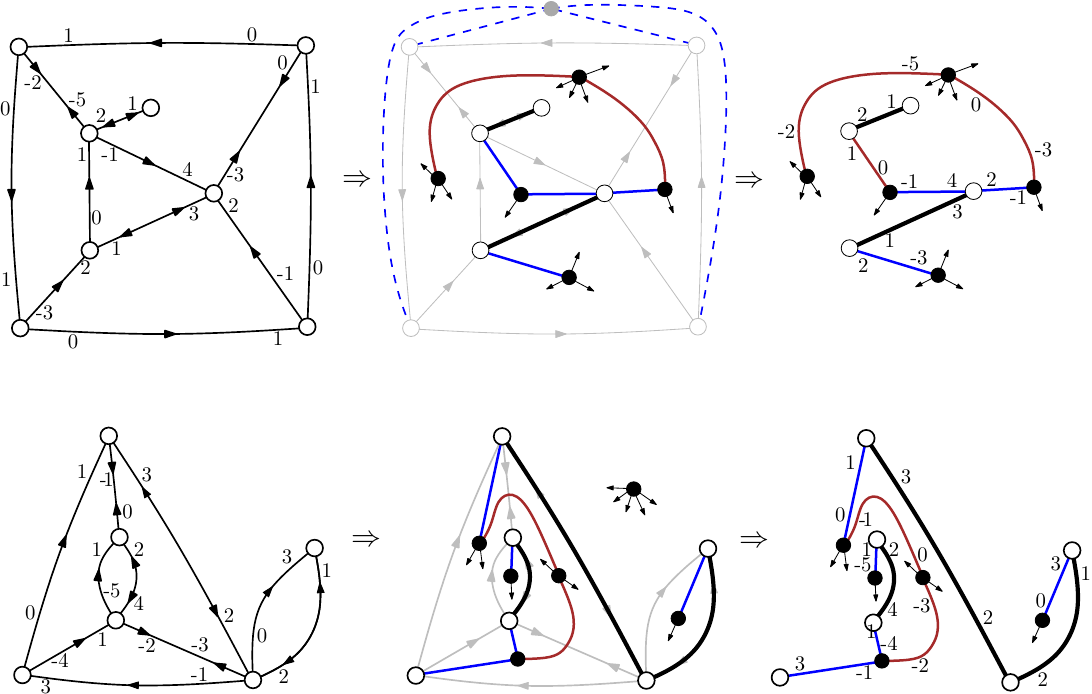}{The master bijection from a $\ZZ$-bioriented plane map in $\cO_{d}$ to a $\ZZ$-mobile of excess $d$ (the top example has $d=-4$, the bottom-example has $d=5$).}

Let $d\in\mathbb{Z}\backslash\{0\}$. We now recall the master bijection $\Phi$ 
introduced in~\cite{BeFu12} between $\cO_d$ and $\cB_d$. For $O\in\cO_d$, we obtain a mobile $T\in\cB_d$ by the following steps 
(see Figure~\ref{fig:master_bij_weighted_biori} for examples):
\begin{enumerate}
\item 
insert a black vertex in each face (including the outer face) of $O$;
\item
apply the local rule of Figure~\ref{fig:local_rule_edge_biori}
(which involves a transfer of weights) to each edge of $O$;
\item
erase the original edges of $O$ and the black vertex $b$ inserted in the outer face of $O$; if $d>0$ erase also the $d$ buds at $b$, if $d<0$ erase also the $|d|$ outer vertices of $O$ and the $|d|$ edges from $b$ to each of the outer vertices.
\end{enumerate}

\begin{theo}[\cite{BeFu12}]
For $d\in\mathbb{Z}\backslash\{0\}$, the mapping $\Phi$ is a bijection between $\cO_d$ and $\cB_d$. 
\end{theo}

The master bijection has the nice property that several parameters of a $\ZZ$-bioriented map $M\in\cO_d$ 
can be read on the associated $\ZZ$-mobile $T=\Phi(M)$. 
We define the \emph{weight} (resp. the \emph{indegree}) of a vertex $v\in M$ as the total
weight (resp. total number) of ingoing half-edges at $v$, 
and we define the \emph{weight} of a face $f\in M$ as the total
weight of the outgoing half-edges having $f$ on their right. 
For a vertex $v\in T$, we define the \emph{degree} of $v$ as the number of half-edges incident to $v$
(including buds if $v$ is black), and we define the \emph{weight} of $v$ as the total weight of the half-edges (excluding buds) incident to $v$. 
 It is easy to see that if $M\in\cO_d$ and $T=\Phi(M)$, then 
\begin{compactitem}
\item every inner face of $M$ corresponds to a black vertex in $T$ of same degree and same weight,
\item for $d>0$ (resp. $d<0$), every vertex (resp. every inner vertex) $v\in M$ corresponds to a white vertex $v'\in T$ of the same weight and such that the indegree of $v$ equals the degree of $v'$. 
\end{compactitem}

\subsection{Adaptation of the master bijection to maps with boundaries}
A face $f$ of a map is said to be \emph{simple} if the number of vertices
incident to $f$ is equal to the degree of $f$ (in other words 
 there is no pair of corners of $f$ incident
to the same vertex). 
A \emph{map with boundaries} is a map $M$ where the set of faces is partitioned 
into two subsets: \emph{boundary faces} 
and \emph{internal faces}, with the constraint that the boundary faces are simple,
and the contours of any two boundary faces are vertex-disjoint; these contour-cycles are called the \emph{boundaries} of $M$. Edges (and similarly half-edges and vertices) are
called \emph{boundary edges} or \emph{internal edges} depending on whether they are on a boundary or not. 
If $M$ is a \emph{plane} map with boundaries, 
whose outer face is a boundary face, then the 
contour of the outer face is called the \emph{outer boundary} and the contours of the other boundary faces
are called \emph{inner boundaries}. 



For $M$ a map with boundaries, a \emph{$\ZZ$-biorientation} of $M$ is called
\emph{consistent} if the boundary edges are all 1-way with weights $(0,1)$ and have 
the incident boundary face on their right. 
For $d\in\mathbb{Z}\backslash\{0\}$, we 
denote by $\wOd$ the set of plane maps with boundaries
endowed with a consistent 
$\ZZ$-biorientation, such that the outer face is a boundary face for $d<0$ and
an internal face for $d>0$, and when forgetting which faces are boundary faces, the underlying $\ZZ$-bioriented plane map is in $\cO_d$. 

A \emph{boundary mobile} is a mobile where every corner at a white vertex might carry additional dangling half-edges
 called \emph{legs}. White vertices having at least one leg are called \emph{boundary vertices}. 
The \emph{degree} of a white vertex $v$ is the number of 
non-leg half-edges incident to $v$. 
The \emph{excess} of a boundary mobile is defined as the number of half-edges incident to a white vertex (including the legs) minus the number of buds. 
A \emph{boundary $\ZZ$-mobile} is a boundary mobile where the 
half-edges different from buds and legs carry weights in $\ZZ$
such that half-edges at white vertices have positive weights
while half-edges at black vertices have non-positive weights. 
For $d\in\mathbb{Z}$, we denote by $\wBd$ the set of boundary $\ZZ$-mobiles of excess $d$.

\fig{width=.8\linewidth}{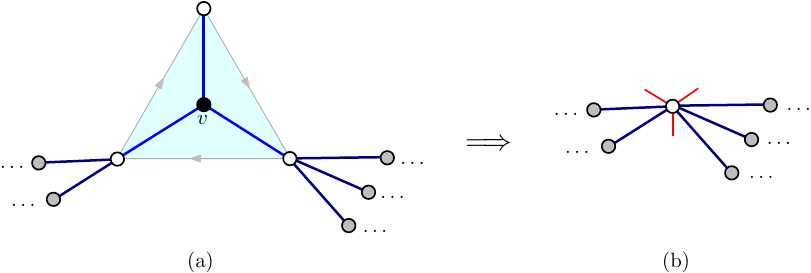}{Reduction operation at the black vertex $v$
corresponding to an inner boundary face in $O\in\cOw_d$.} 

We can now specialize the master bijection. For $O\in\wOd$, let $T=\Phi(O)$ be the associated $\ZZ$-mobile. Note that each inner boundary face $f$ of $O$ of degree $k$ yields a black vertex $v$ of degree $k$ in $T$ such that $v$ has no bud, and the $k$ neighbors $w_1,\ldots,w_k$ of $v$ are the white vertices corresponding to the vertices around $f$. 
We perform the following operation represented in Figure~\ref{fig:reduce}: we insert one leg at each corner of $v$, then contract the edges incident to $v$, and finally recolor $b$ as white.
Doing this for each inner boundary we obtain (without loss of information) a boundary $\ZZ$-mobile $T'$ of the same excess as $T$, called the \emph{reduction} of $T$. 
We denote by $\hPhi$ the mapping such that $\hPhi(O)=T'$. 

We now argue that $\hPhi$ is a bijection between $\wO_d$ and $\wB_d$. For a boundary mobile $T'$, the \emph{expansion} of $T'$ is the mobile $T$ obtained from $T'$ by applying to every boundary vertex the process of Figure~\ref{fig:reduce} in reverse direction: a boundary vertex with $k$ legs yields in $T$ a distinguished black vertex of degree $k$ with no buds, and with only white neighbors. 
Note that, if $T'$ has non-zero excess $d$ and if $O\in\cO_d$ denotes the $\ZZ$-bioriented plane map associated to $T$ by the 
master bijection, then each distinguished face $f\in O$ (i.e., a face associated to a distinguished black vertex of $T$) is simple; indeed if $k\geq 1$ denotes the degree of $f$, the corresponding black vertex $v\in T$ has $k$ white neighbors, which thus correspond to $k$ distinct vertices incident to $f$. In addition the contours 
of the distinguished inner faces are disjoint since the expansions of any two distinct
boundary vertices of $T'$ are vertex-disjoint in $T$. 
Lastly, for $d\in\mathbb{Z}_-^*$,
the outer face is simple and disjoint from the contours of the 
inner distinguished faces (indeed the vertices around an inner distinguished
face of $O$ are all present in $T$, hence are inner vertices of $O$). 
We thus conclude that $O$ belongs to $\cOw_d$, 
upon seeing the distinguished faces (including the outer face 
for $d\in\mathbb{Z}_-^*$) as boundary faces. The following 
statement summarizes the previous discussion: 

\begin{theo}\label{theo:master_boundary}
The master bijection $\hPhi$ adapted to consistent $\ZZ$-biorientations is a bijection between $\wOd$ and $\wBd$ for each $d\in\mathbb{Z}\backslash \{0\}$. 
\end{theo} 

\fig{width=\linewidth}{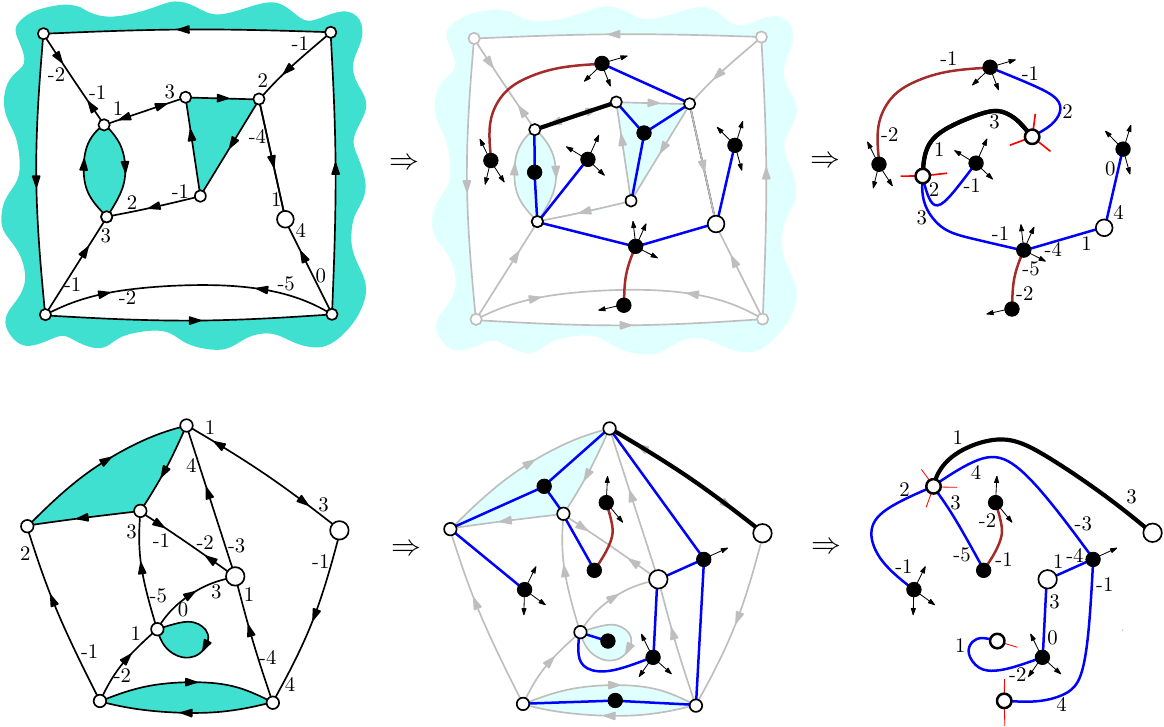}{The master bijection $\hPhi$ applied to two $\ZZ$-bioriented plane maps in $\cOw_{d}$, with $d=-4$ for the top-example, and $d=5$ for the bottom-example. The weights of boundary-edges, which are always $(0,1)$ by definition, are not indicated.}

The bijection $\hPhi$ is illustrated in Figure~\ref{fig:master_bij_weighted_biori_boundary}.
As before, several parameters can be tracked through the bijection. For a map $M$ with boundaries endowed with a consistent $\ZZ$-biorientation, we define the \emph{weight} (resp. the \emph{indegree}) of a boundary $C$ as the total weight (resp. total number) of ingoing half-edges incident to a vertex of $C$ but not lying on an edge of $C$. For a boundary $\ZZ$-mobile, we define the \emph{weight} of a white vertex $v$ as the total weight of the half-edges (excluding legs) incident to $v$. It is easy to see that if $O\in\wOd$ and $T=\hPhi(O)$, then 
\begin{compactitem}
\item every internal inner face of $O$ corresponds to a black vertex in $T$ of same degree and same weight, 
\item every internal vertex $v\in O$ corresponds to a non-boundary white vertex $v'\in T$ of the same weight and such that the indegree of $v$ equals the degree of $v'$, 
\item every inner boundary of length $k$, indegree $r$, and weight $j$ in $O$ corresponds to a boundary vertex in $T$ with $k$ legs, degree $r$, and weight $j$. 
\end{compactitem}

\section{Bijections for quadrangulations with boundaries}\label{sec:bij-quad}
In this section we obtain bijections for \emph{quadrangulations with boundaries}, that is, maps with boundaries such that every internal face
has degree $4$. We start with the simpler case where one of the boundaries has degree 2 before treating the general case.

\subsection{Quadrangulations with at least one boundary of length $2$}\label{sec:outer_degree_2}
We denote by $\cDqua$ the class of bipartite quadrangulations with boundaries, with a marked boundary face of degree 2. We think of maps in $\cDqua$ as \emph{plane maps} by taking the marked boundary as the outer face. 
For $M\in\cDqua$, we call \emph{1-orientation} of $M$ a consistent $\ZZ$-biorientation with weights in $\{-1,0,1\}$ such that:
\begin{compactitem}
\item every internal edge has weight $0$ (hence is either $0$-way with weights $(0,0)$ or 1-way with weights $(-1,1)$), 
\item every internal face (of degree $4$) has weight $-1$,
\item every internal vertex has weight (and indegree) $1$, 
\item every inner boundary of length $2i$ has weight (and indegree) $i+1$, and the outer boundary (of length $2$) has weight (and indegree) $0$.
\end{compactitem}

\begin{prop}\label{prop:2outerquad}
Every map $M\in\cDqua$ has a unique 1-orientation in $\wO_{-2}$. We call it its \emph{canonical biorientation}. 
\end{prop}
The proof of Proposition \ref{prop:2outerquad} is delayed to Section~\ref{sec:proofs}. 
We denote by $\cTqua$ the set of boundary mobiles associated to 
maps in $\cDqua$ (endowed with their canonical biorientation) via the master bijection for maps with boundaries. By  Theorem~\ref{theo:master_boundary}, these are the boundary mobiles with weights in $\{-1,0,1\}$ satisfying the following properties: 
\begin{compactitem}
\item every edge has weight $0$ (hence, is either black-black of weights $(0,0)$, or black-white of weights $(-1,1)$), 
\item every black vertex has degree $4$ and weight $-1$ (hence has a unique white neighbor),
\item for all $i\geq 0$, every white vertex of degree $i+1$ carries $2i$ legs. 
\end{compactitem}
We omit the condition that the excess is $-2$, because it can easily be checked to be a consequence of the above properties.

\fig{width=\linewidth}{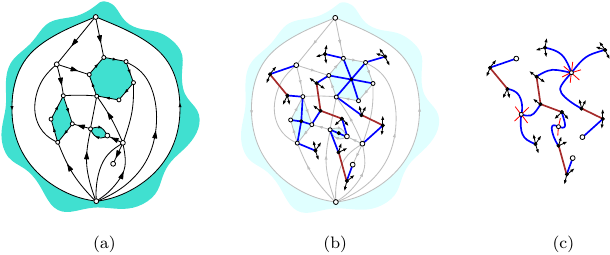}{(a) A map in $\cDqua$ endowed with its canonical biorientation (the 1-way edges are indicated as directed edges, the 0-way edges are indicated as undirected edges, and the weights, which are uniquely induced by the biorientation, are not indicated). (b) The quadrangulation superimposed with the corresponding mobile. (c) The reduced boundary mobile (with $2i$ legs at each white vertex of degree $i+1$), where again the weights (which are uniquely induced by the mobile) are not indicated.}

To summarize, Theorem~\ref{theo:master_boundary} and Proposition~\ref{prop:2outerquad} yield the following bijection (illustrated in Figure~\ref{fig:bij_diss_quad_2outer}) for bipartite quadrangulations with a distinguished boundary of length 2.
\begin{theo}\label{theo:bipartite_2outerdiss}
The set $\cDqua$ of quadrangulations with boundaries is in bijection with the set $\cTqua$ of $\ZZ$-mobiles 
via the master bijection $\hPhi$. 
If $M\in\cDqua$ and $T\in\cTqua$ are associated by the bijection, then each inner boundary of length $2i$ in $M$ corresponds to a white vertex in $T$ of weight (and degree) $i+1$, and each internal vertex of $M$ corresponds to a white leaf in $T$.
\end{theo}




\subsection{Quadrangulations with arbitrary boundary lengths}\label{sec:general_case_quad_diss} 
For $a\geq 1$, we denote by $\Qu{2a}$ the set of bipartite quadrangulations with boundaries with a marked boundary face of degree $2a$. In the previous section we obtained a bijection for $\Qu{2}=\cDqua$. In order to get a bijection for $\Qu{2a}$ when $a>1$, we will need to first mark an edge and decompose our marked maps into two pieces before applying the master bijection to each piece\footnote{A similar strategy was already used in~\cite{BeFu12,BeFu12b,BeFu13}.}.

Let $\Qb{2a}$ be the set of maps obtained from maps in $\Qu{2a}$ by also marking an edge (either an internal edge or a boundary edge). Let $\Qbb{2a}$ be the set of bipartite maps with a marked boundary face of degree $2a$ and a marked internal face of degree 2, such that all the non-marked internal faces have degree 4. We also denote by $\Qvv{2a}$  the set of maps obtained from maps in $\Qbb{2a}$ by marking a corner in the marked boundary face. 

Given a map $M$ in $\Qb{2a}$, we obtain a map $M'$ in $\Qbb{2a}$ by opening the marked edge into an internal face of degree 2. This operation, which we call \emph{edge-opening} is clearly a bijection for $a>1$:
\begin{lem}\label{lem:edge_blow}
For all $a>1$, the edge-opening is a bijection between $\Qb{2a}$ and $\Qbb{2a}$  which preserves the number of internal vertices and the boundary lengths.
\end{lem}
Note however that $\Qbb{2}$ contains a map $\eps$ with 2 edges (a 2-cycle separating a boundary and an internal face) which is not obtained from a map in $\Qb{2}$, so that the bijection is between $\Qb{2}$ and $\Qbb{2}\setminus \{\eps\}$.

We will now describe a canonical decomposition of maps in $\Qbb{2a}$ illustrated in Figure~\ref{fig:bij_quad_diss_complete}(a)-(b). Let $M$ be in $\Qbb{2a}$, and let $f_s$ be the marked boundary face. 
Let $C$ be a simple cycle of $M$, and let $R_C$ and $L_C$ be the regions bounded by $C$ containing $f_s$ and not containing $f_s$ respectively. The cycle $C$ is said to be \emph{blocking} if $C$ has length $2$, the marked internal face is in $L_C$, and any boundary face incident to a vertex of $C$ is in $R_C$. 
Note that the contour of the marked internal face is a blocking cycle. It is easy to see that there exists a unique blocking cycle $C$ such that $L_C$ is maximal (that is, contains $L_{C'}$ for any blocking cycle $C'$). We call $C$ the \emph{maximal blocking cycle} of $M$. The maximal blocking cycle is indicated in Figure~\ref{fig:bij_quad_diss_complete}(a). The map $M$ is called \emph{reduced} if its maximal blocking cycle is the contour of the marked internal face, and we denote by $\Qbbr{2a}$ and $\Qvvr{2a}$ the subsets of $\Qbb{2a}$ and $\Qvv{2a}$ corresponding to reduced maps.

\fig{width=\linewidth}{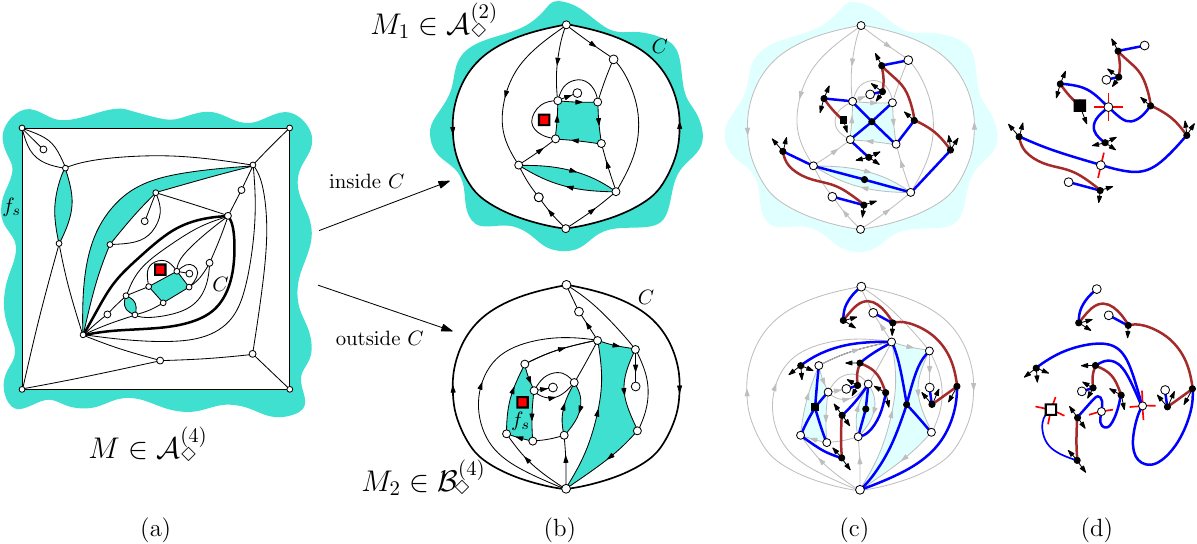}{(a) A map in $\Qbb{4}$: the marked boundary face is the outer face, the marked internal face is indicated by a square, and the maximal blocking cycle $C$ is drawn in bold. (b) The maps $M_1$ and $M_2$ resulting from cutting $M$ along $C$, each represented as a plane map endowed with its canonical biorientation (the marked inner face in each case is indicated by a square). (c) The mobiles associated to $M_1$ and $M_2$. (d) The reduced boundary mobiles associated to $M_1$ and $M_2$, where the marked vertex (corresponding to the marked inner face) is indicated by a square.}

We now consider the two maps obtained from a map $M$ in $\Qbb{2a}$ by ``cutting the sphere'' along the maximal blocking cycle $C$, as illustrated in Figure~\ref{fig:bij_quad_diss_complete}(b). 
Precisely, we denote by $M_1$ the map obtained from $M$ by replacing $R_C$ by a single marked boundary face (of degree 2), and we denote by $M_2$ the map obtained from $M$ by replacing $L_C$ by a single marked internal face (of degree 2). It is clear that $M_1$ is in $\Qbb{2}$, while $M_2$ is in $\Qbbr{2a}$.
Conversely, if we glue the marked boundary face of a map $N_1\in \Qbb{2}$ to the marked internal face of a reduced map $N_2\in \Qbbr{2a}$, we obtain a map $M\in\Qbb{2a}$ whose maximal blocking cycle is the contour of the glued faces, so that $N_1=M_1$ and $N_2=M_2$. In order to make the preceding decomposition bijective, it is convenient to work with \emph{rooted} maps. Given a map $M$ in $\Qvv{2a}$, we define $M_1$ and $M_2$ as above, except that we mark a corner in the newly created boundary face of $M_1$. In order to fix a convention, we choose the corner of $M_1$ such that the vertices incident to the marked corners of $M_1$ and $M_2$ are in the same block of the bipartition of the vertices of $M$. The decomposition $M\mapsto (M_1,M_2)$ is now bijective and we call it the \emph{canonical decomposition} of the maps in $\Qvv{2a}$. We summarize the above discussion:

\begin{lem}\label{lem:decompose-Qvv}
For all $a\geq 1$, the canonical decomposition is a bijection between $\Qvv{2a}$ and $\Qvv{2}\times\Qvvr{2a}$.
\end{lem}
Note that the case $a=1$ above is special in that the set $\Qvvr{2}$ contains only the map $\{\eps\}$. 

Next, we describe bijections for maps in $\Qvv{2}$ and $\Qvvr{2a}$ by using a ``master bijection'' approach illustrated in Figure~\ref{fig:bij_quad_diss_complete}(b)-(d).
For $M\in\Qbb{2}$, we call \emph{1-orientation} of $M$ a consistent $\ZZ$-biorientation of $M$ with weights in $\{-1,0,1\}$ such that: 
\begin{compactitem}
\item every internal edge has weight $0$,
\item every internal vertex has indegree $1$,
\item every non-marked internal face (of degree $4$) has weight $-1$, while the marked internal face (of degree $2$) has weight $0$,
\item every non-marked boundary of length $2i$ has weight (and indegree) $i+1$, while the marked boundary (of length $2$) has weight (and indegree) $0$. 
\end{compactitem}

\begin{prop}\label{prop:Equad}
Let $M$ be a map in $\Qbb{2}$ considered as a plane map by taking the outer face to be the marked boundary face.
Then $M$ admits a unique 1-orientation in $\wO_{-2}$. We call it the \emph{canonical biorientation} of $M$. 
\end{prop}
\begin{proof}
This is a corollary of Proposition~\ref{prop:2outerquad}. Indeed, 
 seeing $M$ as a map $D$ in $\cDqua$ where an edge $e$ is opened into an internal face $f_1$ of degree $2$, the canonical biorientation of $M$
is directly derived from the canonical biorientation of $D$, using the rules shown in Figure~\ref{fig:blow2face}. 
\end{proof}

\fig{width=.7\linewidth}{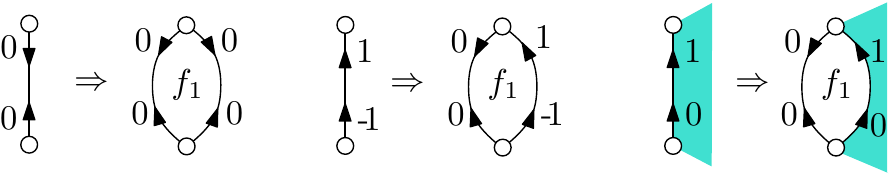}{Transferring the biorientations and weights when blowing an edge into an internal face of degree $2$.}

For $M\in\Qbb{2a}$, we call \emph{1-orientation} of $M$ a consistent $\ZZ$-biorientation with weights in $\{-1,0,1\}$ such that: 
\begin{compactitem}
\item every internal edge has weight $0$, 
\item every internal vertex has weight (and indegree) $1$,
\item every non-marked internal face (of degree $4$) has weight $-1$, while the marked internal face (of degree $2$) has weight $0$, 
\item every non-marked boundary of length $2i$ has weight (and indegree) $i+1$, while the marked boundary (of length $2a$) has weight (and indegree) $a-1$.
\end{compactitem}

\begin{prop}\label{prop2}
Let $M$ be a map in $\Qbb{2a}$ considered as a plane map by taking the outer face to be the marked internal face. Then $M$ has a 1-orientation in $\cOw_2$ if and only if it is reduced (i.e., is in $\Qbbr{2a}$). In this case, $M$ has a unique 1-orientation in $\cOw_2$. We call it the \emph{canonical biorientation} of $M$. 
\end{prop}
The proof of Proposition~\ref{prop2} is delayed to Section~\ref{sec:proofs}. We denote by $\cUqua$ the set of mobiles corresponding to (canonically oriented) maps in $\Qbb{2}$ via the master bijection. By Theorem~\ref{theo:master_boundary}, these are the boundary $\ZZ$-mobiles with weights in $\{-1,0,1\}$ satisfying the following properties (which imply that the excess is $-2$):
\begin{compactitem}
\item every edge has weight $0$ (hence, is either black-black of weights $(0,0)$, or black-white of weights $(-1,1)$), 
\item every black vertex has degree $4$ and weight $-1$ (hence has a unique white neighbor), except for a unique black vertex of degree $2$ and weight $0$,
\item for all $i\geq 0$, every white vertex of degree $i+1$ carries $2i$ legs. 
\end{compactitem}
We also denote $\cUquav$ the set of mobiles obtained from mobiles in $\cUqua$ by marking one of the corners of the black vertex of degree $2$.

For $a\geq 1$, we denote by $\cVqua^{(2a)}$ the set of mobiles corresponding to (canonically oriented) maps in $\Qbbr{2a}$. These are the boundary $\ZZ$-mobiles with weights in $\{-1,0,1\}$ satisfying the following properties (which imply that the excess is $2$):
\begin{compactitem}
\item every edge has weight $0$ (hence is either black-black of weights $(0,0)$, or black-white of weights $(-1,1)$), 
\item every black vertex has degree $4$ and weight $-1$ (hence has a unique white neighbor), 
\item there is a marked white vertex of degree $a-1$ which carries $2a$ legs,
\item for all $i\geq 0$, every non-marked white vertex of degree $i+1$ carries $2i$ legs.
\end{compactitem}
We also denote $\cVquav^{(2a)}$ the set of rooted mobiles obtained from from mobiles in $\cVqua^{(2a)}$ by marking one of the $2a$ legs of the marked white vertex.

Propositions~\ref{prop:Equad} and~\ref{prop2} together with the master bijection (Theorem~\ref{theo:master_boundary}) and Lemma~\ref{lem:decompose-Qvv} finally give:
\begin{theo}\label{theo:bipartite_2annulardiss}
The set $\Qbb{2}$ (resp. $\Qvv{2}$) of quadrangulations is in bijection with the set $\cUqua$ (resp. $\cUquav$) of $\ZZ$-mobiles. 
Similarly, for all $a\geq 1$, the set $\Qbbr{2a}$ (resp. $\Qvvr{2a}$) of quadrangulations is in bijection with the set $\cVqua^{(2a)}$ (resp. $\cVquav^{(2a)}$) of $\ZZ$-mobiles. 

Finally, the set $\Qvv{2a}$ of quadrangulations is in bijection with the set $\cUquav\times\cVquav^{(2a)}$ of pairs of $\ZZ$-mobiles. The bijection is such that if the map $M$ corresponds to the pair of $\ZZ$-mobiles $(U,V)$, then each non-marked boundary of length $2i$ in $M$ corresponds to a non-marked white vertex of $U\cup V$ of weight (and degree) $i+1$, and each internal vertex of $M$ corresponds to a non-marked white leaf of $U\cup V$.
\end{theo}

Theorem~\ref{theo:bipartite_2annulardiss} is illustrated in Figure~\ref{fig:bij_quad_diss_complete}.

\section{Bijections for triangulations with boundaries}\label{sec:bij-triang}
In this section we adapt the strategy of Section~\ref{sec:bij-quad} to  \emph{triangulations with boundaries}, that is, maps with boundaries such that every internal face has degree $3$. We start with the simpler case where one of the boundaries has degree 1 before treating the general case.

\subsection{Triangulations with at least one boundary of length $1$}\label{sec:triang_1}
Let $\cDtri$ be the set of triangulations with boundaries, with a marked boundary face of degree 1. We think of maps in $\cDtri$ as \emph{plane maps} by taking the marked boundary as the outer face. For $M\in\cDtri$, we call \emph{1-orientation} of $M$ a consistent $\ZZ$-biorientation with weights in $\{-2,-1,0,1\}$ and with the following properties:
\begin{compactitem}
\item every internal edge has weight $-1$ (i.e., is either $0$-way of weights $(-1,0)$, or $1$-way of weights $(-2,1)$), 
\item every internal vertex has weight $1$,
\item every internal face has weight $-2$,
\item every inner boundary of length $i$ has weight (and indegree) $i+1$, and the outer boundary has weight~$0$.
\end{compactitem}
Similarly as in Section~\ref{sec:outer_degree_2} we have the following proposition proved in Section~\ref{sec:proofs}.

\begin{prop}\label{prop3}
Every $M\in\cDtri$ has a unique 1-orientation in $\wO_{-1}$. We call it the \emph{canonical biorientation} of $M$. 
\end{prop}

We denote by $\cTtri$ the set of mobiles corresponding to (canonically oriented) maps in $\cDtri$ via the master bijection. By Theorem~\ref{theo:master_boundary}, these are the boundary $\ZZ$-mobiles satisfying the following properties (which readily imply that the weights are in $\{-2,-1,0,1\}$, and the excess is $-1$): 
\begin{compactitem}
\item every edge has weight $-1$ (hence is either black-black of weights $(-1,0)$, or is black-white of weights $(-2,1)$),
\item every black vertex has degree $3$ and weight $-2$,
\item for all $i\geq 0$, every white vertex of degree $i+1$ carries $i$ legs.
\end{compactitem}

\fig{width=\linewidth}{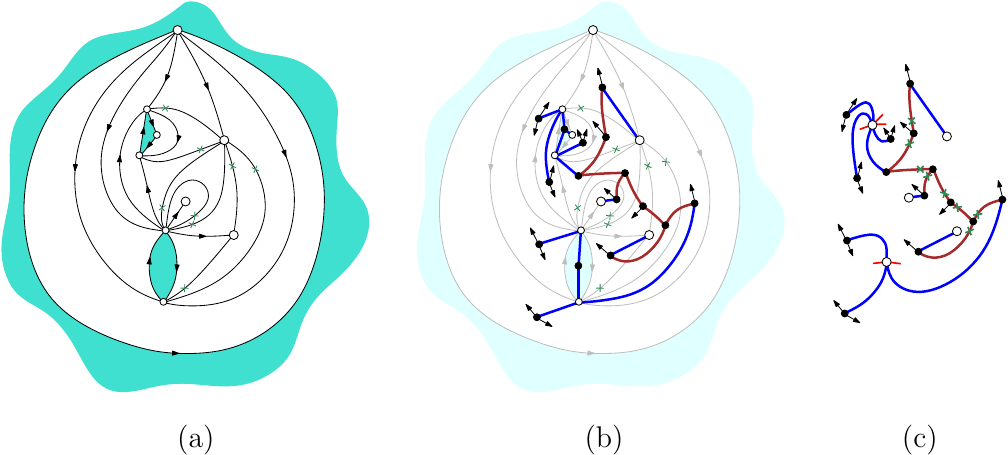}{(a) A triangulation in $\cDtri$ endowed with its canonical biorientation, where crosses indicate half-edges of weight $-1$ (1-way edges have weights $(-1,2)$ if internal and weights $(0,1)$ if boundary, 0-way edges have weights $(-1,0)$). 
(b) The triangulation superimposed with the corresponding mobile. (c) The reduced boundary mobile (with $i$ legs at each white vertex of degree $i+1$, and with again 
the convention that half-edges of weight $-1$ are indicated by a cross).} 

To summarize, we obtain the following bijection for triangulations with a boundary of length 1 (see Figure~\ref{fig:bij_diss_triang_1outer} for an example):
\begin{theo}\label{theo:1outerdiss}
The set $\cDtri$ is in bijection with the set $\cTtri$ via the master bijection. 
If $M\in\cDtri$ and $T\in\cTtri$ are associated by the bijection, then each inner boundary of length $i$ in $M$ corresponds to a white vertex in $T$ of degree $i+1$, and each internal vertex of $M$ corresponds to a white leaf in $T$.
\end{theo}

\subsection{Triangulations with arbitrary boundary lengths}\label{sec:general_case_tri_diss}
We now adapt the approach of Section~\ref{sec:general_case_quad_diss} (decomposing maps into two pieces) to triangulations. 
For $a\geq 1$, we denote by $\Tr{a}$ the set of triangulations with boundaries with a marked boundary face of degree $a$.
We denote by $\Tb{a}$ the set of maps obtained from maps in $\Tr{a}$ by also marking an arbitrary half-edge (either boundary or internal).
We denote by $\Tbb{a}$ the set of maps with boundaries having a marked boundary face of degree $a$ and a marked internal face of degree 1, such that all the non-marked internal faces have degree 3. Lastly, we denote $\Tvv{a}$ the set of maps obtained from $\Tbb{2a}$ by marking a corner in the marked boundary face. 

 Given a map $M$ in $\Tb{a}$, we obtain a map $M'$ in $\Tbb{a}$ by the operation illustrated in Figure~\ref{fig:open_half_edge}, which we call \emph{half-edge-opening}. In words, we ``open'' the edge containing the marked half-edge $h$ into a face $f$, and then at the corner of $f$ corresponding to $h$ we insert a loop bounding the marked internal face (of degree 1). This operation is clearly a bijection for $a>1$:
\begin{lem}\label{lem:edge_blow_tri}
For all $a>1$, the half-edge-opening is a bijection between $\Tb{a}$ and $\Tbb{a}$  which preserves the number of internal vertices and the boundary lengths.
\end{lem}
Note however that $\Tbb{1}$ contains a map $\lambda$ with 1 edges (a loop separating a boundary and an internal face) which is not obtained from a map in $\Tb{1}$, so that the bijection is between $\Tb{1}$ and $\Tbb{1}\setminus \{\lambda\}$. 

\fig{width=.6\linewidth}{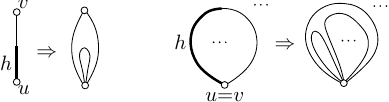}{The operation of opening an half-edge $h$: it yields a new face of degree $1$ surrounded by a new face of degree $3$ (the case where $h$ is on a loop is illustrated on the right).}

Next, we describe a canonical decomposition of maps in $\Tbb{a}$ illustrated in Figure~\ref{fig:bij_tri_diss_complete}(a)-(b). For a cycle $C$ of a map $M\in \Tbb{a}$, we denote by $R_C$ and $L_C$ the regions bounded by $C$ containing and not containing the marked boundary face $f_s$ respectively. The cycle $C$ is said to be \emph{blocking} if $C$ has length $1$ (that is, is a loop), the marked internal face is in $L_C$, and any boundary face incident to a vertex of $C$ is in $R_C$. 
Note that the contour of the marked internal face is a blocking cycle. It is easy to see that there exists a unique blocking cycle $C$ such that $L_C$ is maximal (that is, contains $L_{C'}$ for any blocking cycle $C'$). We call $C$ the \emph{maximal blocking cycle} of $M$. The maximal blocking cycle is indicated in Figure~\ref{fig:bij_tri_diss_complete}(a). The map $M$ is called \emph{reduced} if its maximal blocking cycle is the contour of the marked internal face, and we denote by $\Tbbr{a}$ and $\Tvvr{a}$ the subsets of $\Tbb{a}$ and $\Tvv{a}$ corresponding to reduced maps.

\fig{width=\linewidth}{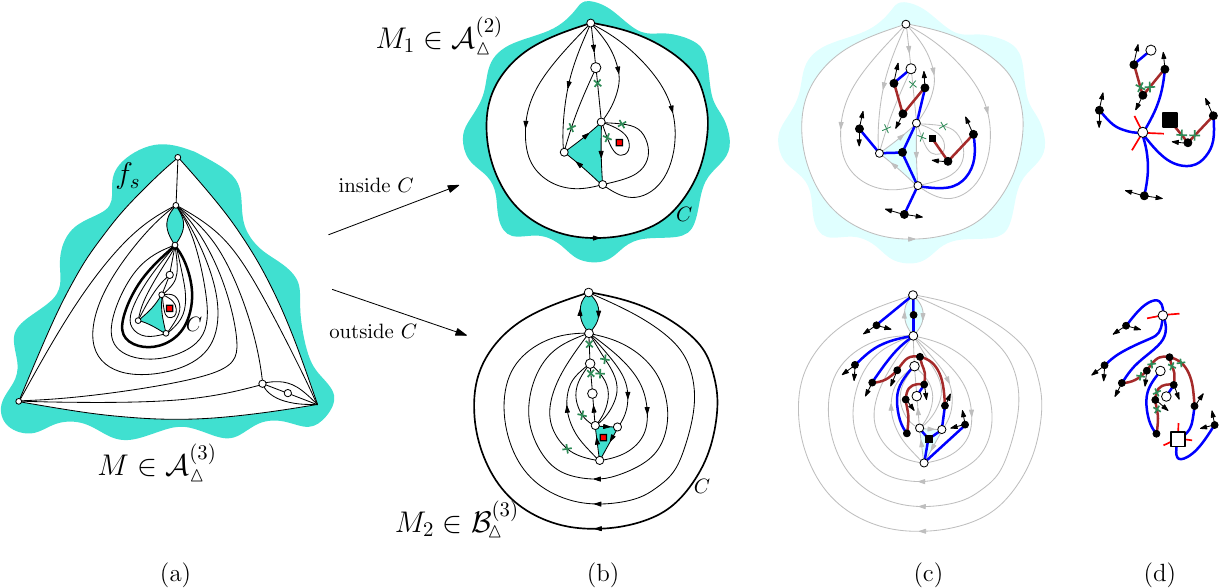}{(a) A map in $\Tbb{3}$: the marked boundary face is the outer face, the marked internal face is indicated by a square, and the maximal blocking cycle $C$ is drawn in bold. (b) The maps $M_1$ and $M_2$ resulting from cutting $M$ along $C$, each represented as a plane map endowed with its canonical biorientation (the marked inner face in each case is indicated by a square, each directed edge has weights $(-2,1)$ and each undirected edge has weights $(-1,0)$, with a cross on the half-edge of weight $-1$). (c) The mobiles associated to $M_1$ and $M_2$. (d) The reduced boundary mobiles associated to $M_1$ and $M_2$, where the marked vertex is represented by a square. Black-white edges
have weights $(-2,1)$, and black-black edges have weights $(-1,0)$, with a cross on the half-edge of weight $-1$.}

We now consider the two maps obtained from a map $M$ in $\Tbb{a}$ by cutting the sphere along the maximal blocking cycle $C$, as illustrated in Figure~\ref{fig:bij_tri_diss_complete}(b).
Precisely, we denote by $M_1$ the map obtained from $M$ by replacing $R_C$ by a single marked boundary face (of degree 1), and we denote by $M_2$ the map obtained from $M$ by replacing $L_C$ by a single marked internal face (of degree 1). It is clear that $M_1$ is in $\Tbb{1}$, while $M_2$ is in $\Tbbr{a}$. The decomposition $M\mapsto (M_1,M_2)$ is bijective (both for rooted and unrooted maps because $\Tvv{1}\simeq \Tbb{1}$), and we call it the \emph{canonical decomposition of maps in $\Tvv{a}$}. We summarize:

\begin{lem}\label{lem:decompose-Tvv}
For all $a\geq 1$, the canonical decomposition is a bijection between $\Tvv{a}$ and $\Tbb{1}\times\Tvvr{a}$.
\end{lem}

Next, we describe bijections for maps in $\Tbb{1}$ and $\Tbbr{a}$ by using the  master bijection approach, as 
illustrated in Figure~\ref{fig:bij_tri_diss_complete}(b)-(d).
For $M\in\Tbb{1}$, we call \emph{1-orientation} of $M$ a consistent $\ZZ$-biorientation of $M$ with weights in $\{-2,-1,0,1\}$ such that: 
\begin{compactitem}
\item every internal edge has weight $-1$,
\item every internal vertex has weight (and indegree) $1$,
\item every non-marked internal face (of degree $3$) has weight $-2$, and the marked internal face (of degree $1$) has weight $0$,
\item every non-marked boundary of length $i$ has weight (and indegree) $i+1$, and the marked boundary has weight (and indegree)~$0$. 
\end{compactitem}

The following result easily follows from  Proposition~\ref{prop3} (similarly as Proposition~\ref{prop:Equad} follows from Proposition~\ref{prop:2outerquad}).
\begin{prop}\label{prop:Etri}
Let $M$ be a map in $\Tbb{1}$, considered as a plane map by taking the marked boundary face as the outer face. 
Then $M$ admits a unique 1-orientation in $\wO_{-1}$. We call it the \emph{canonical biorientation} of $M$. 
\end{prop}

For $M\in\Tbb{a}$ we call \emph{1-orientation} of $M$ a consistent $\ZZ$-biorientation with weights in $\{-2,-1,0,1\}$ such that: 
\begin{compactitem}
\item every internal edge has weight $-1$, 
\item every internal vertex has weight (and indegree) $1$,
\item every internal inner face (of degree $3$) has weight $-2$, and the internal outer face (of degree $1$) has weight $0$. 
\item every non-marked boundary of length $i$ has weight (and indegree) $i+1$, while the marked boundary of length $a$, has weight (and indegree) $a-1$.
\end{compactitem}
\begin{prop}\label{prop4}
Let $M$ be a map in $\Tbb{a}$ considered as a plane map by taking the outer face to be the marked internal face. Then $M$ has a 1-orientation in $\cOw_1$ if and only if it is reduced (i.e., is in $\Tbbr{a}$). In this case, $M$ has a unique 1-orientation in $\cOw_1$, which we call the \emph{canonical biorientation} of $M$. 
\end{prop}
\ni Again the proof is delayed to Section~\ref{sec:proofs}.

We denote by $\cUtri$ the set of mobiles corresponding to (canonically oriented) maps in $\Tbb{1}$ via the master bijection. These are the boundary $\ZZ$-mobiles with weights in $\{-2,-1,0,1\}$ satisfying the following properties (which imply that the excess is $-1$):
\begin{compactitem}
\item every internal edge has weight $-1$ (hence is either black-black of weights $(-1,0)$, or black-white of weights $(-2,1)$), 
\item every black vertex has degree $3$ and weight $-2$, except for a unique black vertex of degree $1$ and weight $0$,
\item for all $i\geq 0$, every white vertex of degree $i+1$ carries $i$ legs. 
\end{compactitem}
For $a\geq 1$, we denote by $\cVtri^{(a)}$ the set of mobiles corresponding to (canonically oriented) maps in $\Tbbr{a}$. These are the boundary $\ZZ$-mobiles with weights in $\{-2,-1,0,1\}$ satisfying the following properties (which imply that the excess is $1$):
\begin{compactitem}
\item every internal edge has weight $-1$ 
\item every black vertex has degree $3$ and weight $-2$, 
\item there is a marked white vertex of degree $a-1$ which carries $a$ legs,
\item for all $i\geq 0$, every non-marked white vertex of degree $i+1$ carries $i$ legs. 
\end{compactitem}
We also denote $\cVtriv^{(a)}$ the set of mobiles obtained from from mobiles in $\cVtri^{(a)}$ by marking one of the $a$ legs of the marked white vertex.
Propositions~\ref{prop:Equad} and~\ref{prop4} together with the master bijection (Theorem~\ref{theo:master_boundary}) and Lemma~\ref{lem:decompose-Tvv} finally give:

\begin{theo}\label{theo:1annulardiss}
The set $\Tbb{1}$ of triangulations is in bijection with the set $\cUtri$ of $\ZZ$-mobiles. 
For all $a\geq 1$, the set $\Tbbr{a}$ (resp. $\Tvvr{a}$) of triangulations is in bijection with the set $\cVtri^{(a)}$ (resp. $\cVtriv^{(a)}$) of $\ZZ$-mobiles.
Finally the set $\Tvv{a}$ of triangulations is in bijection with the set $\cUtri\times \cVtriv^{(a)}$ of pairs of $\ZZ$-mobiles. The bijection is such that if the map $M$ corresponds to the pair of $\ZZ$-mobiles $(U,V)$, then each non-marked boundary of length $i$ in $M$ corresponds to a non-marked white vertex in $U\cup V$ of weight (and degree) $i+1$, and each internal vertex of $M$ corresponds to a non-marked white leaf in $U\cup V$.
\end{theo}

Theorem~\ref{theo:1annulardiss} is illustrated in Figure~\ref{fig:bij_tri_diss_complete}.






\section{Counting results}\label{sec:counting}
\subsection{Proof of Theorem~\ref{theo:krikun_quad} for quadrangulations with boundaries}\label{sec:proof_counting_quad}
We define a \emph{planted mobile of quadrangulated type} as a tree $P$ obtained as one of the two connected components after cutting a mobile $T\in\cTqua$ in the middle of an edge $e$; the half-edge $h$ of $e$ that belongs to $P$ 
is called the \emph{root half-edge} of $P$, and the vertex incident to $h$
is called the \emph{root-vertex} of $P$. 
The \emph{root-weight} of $P$ is the weight of $h$ in $T$. 
For $j\in\{-1,0,1\}$, let $R_j\equiv R_j(t;z_0,z_1,z_2,\ldots)$ be the generating function of planted mobiles
of quadrangulated type having root-weight $j$, where $t$ is conjugate to the number of buds, 
and $z_i$ is conjugate to the number of white vertices of degree $i+1$ (with $2i$ additional legs) 
for $i\geq 0$. We also denote $R:=t+R_0$. 
The decomposition of planted trees at the root easily implies that the series $\{R_{-1},R_0,R_1\}$ are determined by the following system
\begin{equation}
\left\{
\begin{array}{rcl}
R_{-1}&=&R^3,\\
R_0&=&3R_1R^2,\\
R_1&=&\sum_{i\geq 0}z_i\binom{3i}{i}R_{-1}\ \!\!\!^i,
\end{array}
\right.
\end{equation} 
where (for instance) the factor $\binom{3i}{i}$ in the 3rd line accounts
for the number of ways to place the $2i$ legs when the root-vertex has degree $i+1$
(the root half-edge plus $i$ children), 
and the factor $3$ in the second line accounts for choosing which of the 3 children
of the root-vertex is white.

This gives $R=t+3\sum_{i\geq 0}z_i\binom{3i}{i}R^{3i+2}$, or equivalently,
\begin{equation}\label{eq:R}
R=t\phi(R),\ \ \mathrm{with}\ \phi(y)=\Big(1-3\sum_{i\geq 0}z_i\binom{3i}{i}y^{3i+1}\Big)^{-1}.
\end{equation}

Let $\Uqua\equiv \Uqua(t;z_0,z_1,\ldots)$ 
be the generating function of mobiles in $\cUquav$
with $t$ conjugate to the number of buds and $z_i$ conjugate 
to the number of white vertices of degree $i+1$ for $i\geq 0$. 
For $a\geq 1$, let $\Vqua^{(2a)}\equiv\Vqua^{(2a)}(t;z_0,z_1,\ldots)$ be the generating function of mobiles in $\cVquav^{(2a)}$
with $t$ conjugate to the number of buds and $z_i$ conjugate
to the number of non-marked white vertices of degree $i+1$ for $i\geq 0$. The decomposition at the marked vertex gives 
$$
\Uqua=R^2, \textrm{ and } \Vqua^{(2a)}=\binom{3a-2}{a-1}R_{-1}\ \!\!^{a-1}=\binom{3a-2}{a-1}R^{3a-3}.
$$


Let $\vAqua{2a}\equiv \vAqua{2a}(z_0,z_1,\ldots)$ be the generating function of $\Qvv{2a}$, where $z_0$ is conjugate to the number of internal vertices and for all $i\geq 1$, $z_{i}$ is conjugate to the number of unmarked boundaries of length $2i$. 
 Theorem~\ref{theo:bipartite_2annulardiss} gives
\begin{equation*}
\vAqua{2a}=\Uqua\times\Vqua^{(2a)}|_{t=1}= \binom{3a-2}{a-1}R^{3a-1}|_{t=1}.
\end{equation*}


Now let $\beta_a(m;n_1,\ldots,n_h)$ be the number of maps in $\Qu{2a}$ with a marked corner in the marked boundary face, with $m$ internal vertices, $n_i$ non-marked boundaries of length $2i$ for $1\leq i\leq h$, and no inner 
boundary of length larger than $2h$. The half total boundary length is $b=a+\sum_i in_i$, the total number of boundaries is $r=1+\sum_in_i$. Moreover, by the Euler relation, the number of edges is $e=3b+2r+2m-4=3b+2k$, where $k:=r+m-2$. 
Then Lemma~\ref{lem:edge_blow} yields 
$$
e\ \!\beta_a(m;n_1,\ldots,n_h)=[z_0^m z_1^{n_1}\cdots z_h^{n_h}]\vAqua{2a}=\binom{3a-2}{a-1}[z_0^m z_1^{n_1}\cdots z_h^{n_h}]R^{3a-1}|_{t=1}.
$$
 

It is easy to see from \eqref{eq:R} that the variable $t$ is redundant in $R$, and that for all $q,n_0,\ldots n_h$, 
$$\ds [z_0^{n_0} z_1^{n_1}\cdots z_h^{n_h}]R^q|_{t=1}=[z_0^{n_0} z_1^{n_1}\cdots z_h^{n_h}][t^{q+\sum_{i=0}^h(3i+1)n_i}]R^q.$$
Moreover, by the Lagrange inversion formula~\cite[Thm 5.4.2]{Stan1999}, \eqref{eq:R} implies that for any positive integers $n,q$, 
$$
[t^n]R^q=\frac{q}{n}[y^{n-q}]\phi(y)^n.
$$
Thus, denoting $p:=3a-1+m+\sum_{i=1}^h(3i+1)n_i=m+r+3b-2=k+3b$, we get
\begin{eqnarray*}
[z_0^{m}z_1^{n_1}\cdots z_h^{n_h}]R^{3a-1}|_{t=1}&=&[z_0^{m}z_1^{n_1}\ldots z_h^{n_h}][t^p]R^{3a-1}\\
&=&\frac{3a-1}{p}[z_0^m\cdots z_h^{n_h}][y^{p-3a+1}]\Big(1-3\sum_{i=0}^h z_i\binom{3i}{i}y^{3i+1}\Big)^{-p}\\
&=&\frac{3a-1}{p}[z_0^m\cdots z_h^{n_h}]\Big(1-3\sum_{i=0}^h z_i\binom{3i}{i}\Big)^{-p}\\
&=&\frac{3a-1}{p}3^{m+r-1}\binom{p-1+m+r-1}{p-1,m,n_1,\ldots,n_h}\prod_{i=1}^h\binom{3i}{i}^{n_i}.\\
\end{eqnarray*}
Using $k=m+r-2$, $p=k+3b$, $e=p+k$, and $(3a-1)\binom{3a-2}{a-1}=\tfrac{2}{3}a\binom{3a}{a}$, we get
\begin{equation}\label{eq:beta_a_quad}
\beta_a(m;n_1,n_2,\ldots,n_h)=3^{k}\frac{(e-1)!}{m!(k+3b)!}2a\binom{3a}{a}\prod_{i=1}^h\frac{1}{n_i!}\binom{3i}{i}^{n_i},
\end{equation}
which, multiplied by $\prod_{i=1}^h n_i!(2i)^{n_i}$ (to account for numbering the 
inner boundary faces and marking a corner in each of these faces), gives~\eqref{eq:krikun_quad}.

\subsection{Proof of Theorem~\ref{theo:krikun_triang} for triangulations with boundaries}\label{sec:proof_counting_tri}
We proceed similarly as in Section~\ref{sec:proof_counting_quad}. 
We call \emph{planted mobile of triangulated type} any tree $P$ equal to one of the two connected components obtained from some $T\in\cTtri$ by cutting an edge $e$ in its middle; the half-edge $h$ of $e$ belonging to $P$ is called the \emph{root half-edge} of $P$, and the weight of $h$ in $T$
is called the \emph{root-weight} of $P$. 
For $j\in\{-2,-1,0,1\}$, let  
$S_{j}\equiv S_j(t;z_0,z_1,\ldots)$ be the generating function of planted mobiles of triangulated type
and root-weight $j$, 
with $t$ conjugate to the number of buds and $z_i$ conjugate to 
the number of white vertices of degree $i+1$ for $i\geq 0$. 
We also define $S:=t+S_{-1}$. The decomposition of planted trees at the root easily implies that the series $\{S_{-2},S_{-1},S_0,S_1\}$ are determined by the following system:
\begin{equation}
\left\{
\begin{array}{rcl}
S_{-2}&=&S^2,\\
S_{-1}&=&2SS_{0},\\
S_0&=&2SS_1+S_0^2,\\
S_1&=&\sum_{i\geq 0}z_i\binom{2i}{i}S_{-2}\ \!\!\!^i.
\end{array}
\right.
\end{equation} 
The second line gives $\ds S_0=\frac{1}{2}(1-\frac{t}{S})$. Hence the third line gives  $S^2=t^2+8S_1S^3$. Moreover the first and fourth line gives $\ds S_1=\sum_{i\geq 0}z_i\binom{2i}{i}S^{2i}$. Thus,
\begin{equation}\label{eq:S}
S=t\phi(S),\ \ \mathrm{where}\ \phi(y)=\Big(1-8\sum_{i\geq 0}z_i\binom{2i}{i}y^{2i+1}\Big)^{-1/2}. 
\end{equation}

Let $\Utri\equiv \Utri(t;z_0,z_1,\ldots)$ 
be the generating function of mobiles from $\cUtri$, with $t$ conjugate to the number of buds and $z_i$ conjugate
to the number of white vertices of degree $i+1$ for $i\geq 0$. 
And for $a\geq 1$ let $\Vtri^{(a)}\equiv\Vtri^{(a)}(t;z_0,z_1,\ldots)$ be the generating function of mobiles in 
 $\cVtriv^{(a)}$, with $t$ conjugate to the number of buds and $z_i$ conjugate
to the number of non-marked white vertices of degree $i+1$ for $i\geq 0$. A decomposition at the marked vertex gives 
$$
\Utri=S, \textrm{ and } \Vtri^{(a)}=\binom{2a-2}{a-1}S_{-2}\ \!\!^{a-1}=\binom{2a-2}{a-1}S^{2a-2}.
$$

Let $\vAtri{a}\equiv \vAtri{a}(z_0,z_1,\ldots)$ be the generating function of $\Tvv{a}$, where $z_0$ is conjugate to the number of internal vertices and for all $i\geq 1$, $z_{i}$ is conjugate to the number unmarked boundaries of length $i$.
Theorem~\ref{theo:1annulardiss} gives 
\begin{equation}
\vAtri{a}=\Utri\times\vBtri{a}= \binom{2a-2}{a-1}S^{2a-1}|_{t=1}.
\end{equation}


We define now $\eta_a(m;n_1,n_2,\ldots,n_h)$ as the number of triangulations with a marked boundary of length $a$ having a marked corner, with $m$ internal vertices, $n_i$ non-marked boundaries of length $n_i$ for $1\leq i\leq h$, and no non-marked boundary of length larger than $h$. 
The total boundary-length is $b:=a+\sum_i in_i$, the number of boundaries is $r=1+\sum_in_i$, and (by the Euler relation)
the number of edges is  $e=2b+3r+3m-6$, which is $2b+3k$ with $k:=r+m-2$. 
Then Lemma~\ref{lem:edge_blow_tri} yields
$$
2e\ \!\eta_a(m;n_1,n_2,\ldots,n_h)=[z_0^m z_1^{n_1}\cdots z_h^{n_h}]\vAtri{a}=
[z_0^m z_1^{n_1}\cdots z_h^{n_h}]\binom{2a-2}{a-1}S^{2a-1}|_{t=1}.
$$

It is easy to see from \eqref{eq:S} that for all positive integers $q,n_0,\ldots,n_h$,
$$
[z_0^{n_0}z_1^{n_1}\cdots z_h^{n_h}]S^q|_{t=1}=[z_0^{n_0}z_1^{n_1}\cdots z_h^{n_h}][t^{q+\sum_{i=0}^h(2i+1)n_i}]S^q.
$$

Hence, by the Lagrange inversion formula, 
and using the notation
$p:=2a-1+m+\sum_{i=1}^h(2i+1)n_i=2b+k$ and $s=m+\sum_{i\geq 1}n_i=k+1$ gives
\begin{eqnarray*}
[z_0^{m}z_1^{n_1}\ldots z_h^{n_h}]S^{2a-1}|_{t=1}&=&
[t^{p}z_0^{m}z_1^{n_1}\ldots z_h^{n_h}]S^{2a-1}\\
&=&\frac{2a-1}{p}[y^{p-2a+1}z_0^{m}z_1^{n_1}\ldots z_h^{n_h}]\Big(1-8\sum_{i=0}^h z_i\binom{2i}{i}y^{2i+1}\Big)^{-p/2}\\
&=&\frac{2a-1}{p}[z_0^{m}z_1^{n_1}\ldots z_h^{n_h}]\Big(1-8\sum_{i=0}^h z_i\binom{2i}{i}\Big)^{-p/2},\\
&=&\frac{2a-1}{p}[z_0^{m}z_1^{n_1}\ldots z_h^{n_h}]\left(8\sum_{i=0}^h z_i\binom{2i}{i}\right)^s\cdot[u^s](1-u)^{-p/2}\\
&=&\frac{2a-1}{p}8^s\binom{s}{m,n_1,\ldots,n_h}\Big(\prod_{i=1}^h\binom{2i}{i}^{n_i}\Big)\cdot \frac{(p+2s-2)!!}{(p-2)!!s!2^s}\\
&=&\frac{2a-1}{m!}4^s\Big(\prod_{i=1}^h\frac1{n_i!}\binom{2i}{i}^{n_i}\Big)\cdot \frac{(p+2s-2)!!}{p!!}.
\end{eqnarray*}
Thus, using $p+2s-2=e$, $s=k+1$, and $p=2b+k$ we get 
$$
\eta_a(m;n_1,n_2,\ldots,n_h)=4^k\frac{(e-2)!!}{m!\ \!(2b+k)!!}a\binom{2a}{a}\prod_{i=1}^h\frac1{n_i!}\binom{2i}{i}^{n_i}.
$$
Multiplying this expression 
by $\prod_{i=1}^h n_i!i^{n_i}$ (to account for numbering the 
inner boundary faces and marking a corner in each of these faces)
gives~\eqref{eq:krikun_triang}.

\subsection{Solution of the dimer model on quadrangulations and triangulations}\label{sec:dimers} 
A \emph{dimer-configuration} on a map $M$ is a subset $X$ of the non-loop edges of $M$ such that every vertex of $M$ is incident to at most one edge in $X$. The edges of $X$ are called \emph{dimers}, and the vertices not incident to a dimer are called \emph{free}. 
The \emph{partition function} of the \emph{dimer model} on a class $\cC$ of maps  is the generating function of maps in $\cC$ endowed with a dimer configuration, counted according to the number of dimers and free vertices. The partition function of the dimer model is known for rooted 4-valent maps~\cite{staudacher1990yang,BoDFGu03} (and more generally $p$-valent maps). 



We observe that counting (rooted) maps with dimer configurations is a special case of counting (rooted) maps with boundaries. More precisely, upon blowing each dimer into a boundary face of degree $2$, a rooted map with a dimer-configuration can 
be seen as a rooted map with boundaries, such that all boundaries have length $2$, and the rooted corner is in an internal face. 
Based on this observation we 
easily obtain from Theorem~\ref{theo:krikun_quad} that, for all $m,r\geq 0$ with $m+2r\geq 3$,
 the number $q_{m,r}$ of dimer-configurations on rooted quadrangulations with $r$ dimers and $m+2r$
vertices is
\begin{equation}\label{eq:dimer_qua}
q_{m,r}=4(m+2r-2)\frac{3^{2r+m-2}(5r+2m-5)!}{r!m!(4r+m-2)!}.
\end{equation} 
Similarly, Theorem~\ref{theo:krikun_triang} implies that, for all $m,r\geq 0$ with $m+2r\geq 3$,
 the number $t_{m,r}$ 
of dimer-configurations on rooted triangulations with $r$ dimers and $m+2r$
vertices is 
\begin{equation}\label{eq:dimer_tri}
t_{m,r}=(m+2r-2)\frac{2^{2m+3r-3}3^{r+1}(7r+3m-8)!!}{r!m!(5r+m-2)!!}.
\end{equation}

In the context of statistical physics it would be useful to have an expression for the partition function, that is, the generating function of the coefficients $q_{m,r}$ or $t_{m,r}$. It should be possible to lift
the expressions in~\eqref{eq:dimer_qua} and~\eqref{eq:dimer_tri} to generating function expressions, 
however we find it easier to obtain directly an exact expression from
the bijections of Section~\ref{sec:outer_degree_2} (for quadrangulations) and Section~\ref{sec:triang_1} (for triangulations), without a possibly technical lift from the coefficient expressions.
 Here this works by considering generating functions for the model with a slight restriction at the root edge. 

For quadrangulations, we consider the generating function $Q(x,w)$ of rooted quadrangulations endowed with a dimer-configuration, with the constraint that both extremities of the root edge are free, where $x$ is conjugate to the number of free vertices minus $2$, and $w$ is conjugate to the number of dimers. These objects are clearly in bijection (by opening the root-edge and every dimer into a boundary face of degree $2$) with the set $\cQ$ of rooted quadrangulation with boundaries all of length $2$, such that the root-corner is in a boundary face. 
So $Q(x,w)$ is the generating function of maps in $\cQ$, where $x$ is conjugate to the number of internal vertices and $w$ is conjugate to the number of inner boundaries. Note that $\cQ$ can be seen as a subset of $\cDqua$, except that we are marking a corner in the outer face. Thus, applying the bijection of Section~\ref{sec:outer_degree_2}, we can interpret $Q(x,w)$ in terms of the set $\cTqua'$ of mobiles from $\cTqua$ such that every boundary vertex has 2 legs. More precisely, upon remembering that mobiles in $\cTqua$ have excess -2,
it is not hard to see that $Q(x,w)=Q_1-Q_2$, where $Q_1$ (resp. $Q_2$) is the generating function of mobiles from $\cTqua'$ with a marked bud (resp. with a marked leg or half-edge at a white vertex) with $x$ counting white leaves, and $w$ counting boundary vertices. 
From the series expressions obtained in Section~\ref{sec:proof_counting_quad} we get 
 $Q_1=R_0=R-1$ and $Q_2=x\ \!R_{-1}+6w\ \!R_{-1}\ \!\!^2$, under 
the specialization $\{t=1,z_0=x,z_1=w,z_i=0\ \forall i\geq 2\}$. 
Hence
\begin{equation}
Q(x,w)=R-1-x\ \!R^3-6w\ \!R^6,\ \ \mathrm{where}\ R=1+3xR^2+9w\ \!R^5.
\end{equation} 
Note that $\widetilde{Q}(z,w):=Q(z,z^2w)$ is the generating 
function for the same objects, with $z$ conjugate to the number of vertices minus $2$ (which by the Euler relation is also the number of faces) 
and $w$ conjugate to the number of dimers. Now, if we are interested in the \emph{phase transition} of this model, we need to determine how the asymptotic behavior of the coefficients $c_n=[z^n]\widetilde{Q}(z,w)$ (for $n\to \infty$) depends on the parameter $w$. 
According to the principles of analytic combinatorics \cite{FlaSe}, we need to study the dominant singularities of $\widetilde{Q}(z,w)$ \emph{considered as a function of $z$}. A maple  worksheet detailing the necessary calculations can be found on the webpages of the authors; we only report the results here. Let $\sigma(w)$ be the dominant
singularity of $\widetilde{Q}(z,w)$, and let $Z=\sigma(w)-z$. 
For all $w\geq 0$, the singularity type of $\widetilde{Q}(z,w)$ is 
$Z^{3/2}$ (as for maps without dimers), and no phase-transition occurs. 
However we find a singular value of $w$ at $w_0=-3/125$, 
where $\sigma(w_0)=4/45$ and the singularity of $\widetilde{Q}(z,w_0)$ is of type $Z^{4/3}$ (as a comparison, it is shown in~\cite[Sec.6.2]{BoDFGu03} that for the dimer model on rooted 4-valent maps, the critical value of the
dimer-weight is $w_0=-1/10$ and the singularity type is the same: $Z^{4/3}$). 
 
For triangulations we consider the generating function $T(x,w)$ of rooted triangulations endowed with a dimer-configuration, with the constraint that the root-vertex is free, where $x$ is conjugate to the number of free vertices minus $1$, and $w$ is conjugate to the number of dimers. These objects are in bijection (up to opening the dimers into boundaries and opening the root half-edge as in Figure~\ref{fig:open_half_edge}) with the set $\cT$ of triangulations with boundaries, with one boundary of degree $1$ taken as the outer face and 
 all the other boundaries (inner boundaries) of length $2$, and such
that there are at least two inner faces. Let $\tau$ be the unique triangulation with one boundary face of length 1 (the outer face) and one inner face. 
By the preceding, $T(x,w)+x$ is the generating function of maps in $\cT'=\cT\cup\{\tau\}$. The bijection of Section~\ref{sec:triang_1} applies to the set $\cT'$ and allows us to express $T(w,x)$ in terms of the set $\cTtri'$ of mobiles from $\cTtri$ such that every boundary vertex has 2 legs. 
More precisely, upon remembering that mobiles in $\cTtri'$ have excess -1, this bijection gives $T(x,w)+x=T_1-T_2$, where $T_1$ (resp. $T_2$) is the generating function of mobiles from $\cTtri'$ with a marked bud (resp. a marked leg or half-edge incident to a white vertex) with $x$ counting white leaves and $w$ counting boundary vertices. From the series expressions obtained in Section~\ref{sec:proof_counting_tri}, we get 
 $T_1=S_0=(S-1)/(2S)$ and $T_2=x\ \!S_{-2}+10w\ \!S_{-2}\ \!\!^3$, under 
the specialization $\{t=1,z_0=x,z_2=w,z_i=0\ \forall i\notin\{0,2\}\}$. 
Hence
\begin{equation}
T(x,w)=\frac{S-1}{2S}-x-x\ \!S^2-10\ \!w\ \!S^6,\ \ \mathrm{where}\ S^2=1+8xS^3+48\ \!w\ \!S^7.
\end{equation} 
Again we note that $\widetilde{T}(z,w):=T(z,z^2w)$ is the generating 
function for the same objects, with $z$ conjugate to the number of vertices minus $1$ (which by the Euler relation is also one plus half the number of faces) and $w$ conjugate to the number of dimers. We now discuss the phase transition. 
We use the notations $\sigma(w)$ for the dominant
singularity of $\widetilde{T}(z,w)$, and $Z=\sigma(w)-z$. We
 find that for all $w\geq 0$, the singularity of $\widetilde{T}(z,w)$ is 
of type $Z^{3/2}$, so that no phase-transition occurs. 
However, we find a singular value $w_0=-8\sqrt{105}/5145\approx -0.0159$, for which $\sigma(w_0)=5\sqrt{105}/1008\approx 0.0508$ 
and $\widetilde{T}(z,w_0)$ has singularity type $Z^{4/3}$. 


\section{Generalization to arbitrary face degrees}\label{sec:girth} 
We present here a unification and extension of the results of Section~\ref{sec:bij-quad} and Section~\ref{sec:bij-triang}. 
In view of the results established in~\cite{BeFu12,BeFu12b,BeFu13}, one could hope to find bijections for all maps  with boundaries of girth at least $d$ (for any fixed $d\geq 1$),  keeping track of the distribution of the internal face degrees and of the boundary face degrees. 
However,  when trying to achieve this goal we met two obstacles. First, the natural parameter we can control through our approach is not the girth but a related notion that we call \emph{\bgirth} (it coincides with the girth when there are at most one boundary; see definition below). Second, in \bgirth\ $d$ we obtained bijections  only in the case where the degrees of internal faces are in $\{d,d+1,d+2\}$. These constraints appear when trying to characterize maps with boundaries by canonical orientations (see Section \ref{sec:proofs}). Nonetheless, the results presented here give a bijective proof to all the known enumerative results for maps with boundaries.

Let us first define the \bgirth. Let $M$ be a map with boundaries. The \emph{contour-length} of a set $S$ of faces of $M$ is the number of edges separating a face in $S$ from a face not in $S$. Note that the \emph{girth} of $M$ (that is, the minimal length of cycles) is equal to the minimal possible contour-length of a non-empty set of faces.
A set $S$ of faces of $M$ is called \emph{internally-enclosed}, if any face sharing a vertex with a boundary face in $S$ is also in $S$. 
For a boundary face $f_s$ of $M$, we define the $f_s$-\emph{\bgirth} of $M$ as the minimal possible contour-length of a non-empty internally-enclosed set of faces not containing $f_s$. Clearly, the $f_s$-\bgirth\ is greater or equal to the girth. 
On the other hand, the $f_s$-\bgirth\ is smaller or equal to the degree of any internal face $f$ (by considering $S=\{f\}$) and to the degree of $f_s$ (by considering the set $S$ of faces distinct from $f_s$).
Moreover, if $M$ has no boundary except for $f_s$, then the $f_s$-\bgirth\ of $M$ coincides with the girth (because any set of faces not containing $f_s$ is internally-enclosed). 

\fig{width=\linewidth}{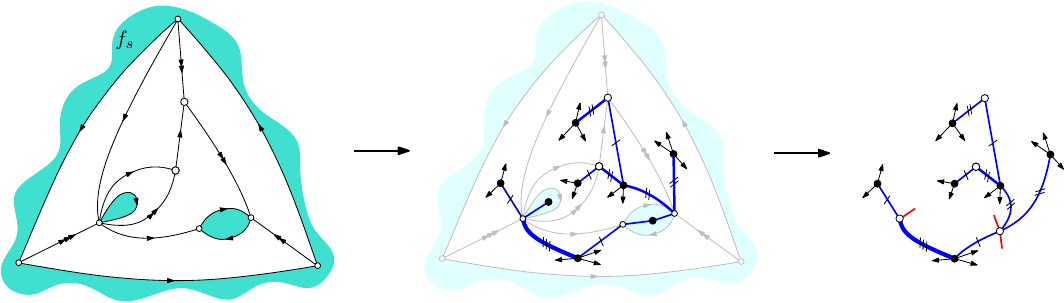}{Left: a map in $\cD_3$
endowed with its canonical orientation (all edges are 1-way
for $d=3$, the weight at the ingoing half-edge is indicated
as the multiplicity of the arrow). Right: the corresponding
mobile in $\cT_3$ (all edges are black-white for $d=3$,  
the weight at the 
white extremity is indicated as the number of stripes across the edge). 
}

\subsection{Bijections in $f_s$-\bgirth\ $d$, when $f_s$ has degree $d$} 
For $d\geq 1$, we define $\cEd$ as the set of 
maps with boundaries having a marked boundary-face $f_s$ of degree $d$, and where every internal face has 
degree in $\{d,d+1,d+2\}$. Clearly, maps in $\cEd$ have $f_s$-\bgirth\ at most $d$, and we denote by $\cDd$ the subset of maps from $\cEd$ having $f_s$-\bgirth\ $d$. For instance, $\cDtri$ is the set of maps $\cE_1=\cD_1$ with all the internal faces of degree 3. Similarly, $\cDqua$ is the set of \emph{bipartite} maps in $\cE_2$ with all the internal faces of degree 4; the bipartiteness condition is equivalent to the fact that every boundary has even length, and implies that the $f_s$-\bgirth\ is 2 (hence $\cDqua\subset \cD_2$).

For $M\in\cEd$, a \emph{$d/(d-2)$-orientation} of $M$ is defined as 
a consistent $\ZZ$-biorientation of $M$ (with weights in 
$\{-2,-1,\ldots,d\}$) such that:
\begin{compactitem}
\item
every internal edge has weight $d-2$,
\item
every internal vertex has weight $d$, 
\item
every internal face $f$ has weight $d-\deg(f)$,
\item
every boundary face $f\neq f_s$ has weight $d+\deg(f)$, while 
the boundary of $f_s$ has weight~$0$. 
\end{compactitem}
Note that the notion of $1$-orientation for maps in $\cDtri$ given in Section \ref{sec:bij-triang} coincides with the notion of $d/(d-2)$-orientation for $d=1$.
Note also that if $X$ is a $1$-orientation of a map in $\cDqua$ as defined in Section \ref{sec:bij-quad}, then multiplying the weights of every internal half-edge of $X$ by $2$ gives a  $d/(d-2)$-orientation for $d=2$.

\begin{prop}\label{prop:Ddori}
Let $M$ be a map in $\cEd$ considered as a plane map by taking the outer face to be the marked boundary face $f_s$.
The map $M$ has a $d/(d-2)$-orientation if and only if $M$ has $f_s$-\bgirth\ at least $d$ (i.e. is in $\cDd$). 
Moreover, in this case, $M$ has a unique $d/(d-2)$-orientation in $\wO_{-d}$. We call it the \emph{canonical biorientation} of $M$. 

In the case where $d$ is even, a map $M\in \cDd$ is bipartite if and only if all the internal half-edges of $M$ have even weight in its canonical biorientation. 
\end{prop}
The proof is delayed to Section~\ref{sec:proofs}. 
Note that Proposition~\ref{prop:Ddori} includes Proposition~\ref{prop3} (case $d=1$, with all internal faces of degree $3$), and also Proposition~\ref{prop:2outerquad} (case $d=2$, with all internal faces of degree $4$) upon dividing all weights on internal half-edges by $2$.

We denote by $\cTd$ the set of mobiles corresponding to (canonically oriented) maps in $\cEd$ via the master bijection. These are the boundary $\ZZ$-mobiles satisfying the following properties (which readily imply that the weight of every half-edge is in $\{-2,-1,\ldots,d\}$ and the excess is $-d$):
\begin{compactitem}
\item every edge has weight $d-2$, 
\item every black vertex has degree $\delta$ in $\{d,d+1,d+2\}$ and weight $-\delta+d$,
\item every white vertex has weight $d+\ell$, where $\ell$ is the number of incident legs.
\end{compactitem}

Proposition~\ref{prop:Ddori} together with the master bijection (Theorem~\ref{theo:master_boundary}) then yield the following result.


\begin{theo}\label{theo:bij_Dd}
For each $d\geq 1$, the set $\cDd$ of maps is in bijection with the set $\cTd$ of $\ZZ$-mobiles via the master bijection.
If a map $M\in\cDd$ corresponds to a mobile $T\in\cTd$ by this bijection, then every 
internal face of $M$ corresponds to a black vertex of $T$ of the same degree, every internal vertex of $M$ corresponds to a white vertex of $T$ with no leg, and every boundary face $f\neq f_s$ in $M$ corresponds 
to a white vertex of $T$ having $\mathrm{deg}(f)$ legs. 
\end{theo}

\subsection{Bijections in \bgirth\ $d$ when at least one internal face has degree $d$}\label{sec:bij_annular} 
As a unification and generalization of Sections~\ref{sec:general_case_quad_diss} and~\ref{sec:general_case_tri_diss}, 
we treat here the case where the marked boundary face $f_s$ has arbitrary degree, but at least one internal face
has degree equal to the $f_s$-\bgirth. 

For $d,a\geq 1$, we denote by $\cGda$ the set of maps with boundaries, with a marked boundary face $f_s$ of degree $a$, and a marked internal face $f_e$ of degree $d$, and where all internal faces have degree in $\{d,d+1,d+2\}$.
Clearly, maps in $\cGda$ have $f_s$-\bgirth\ at most $d$, and we denote by $\cAda$ the subset of maps from $\cGda$ having $f_s$-\bgirth\ $d$. 
For instance, $\Tbb{a}$ is the set of maps $\cG_1^{(a)}=\cA_1^{(a)}$ with all the internal faces of degree 3. 
Similarly, $\Qbb{2a}$ is the set of \emph{bipartite} maps in  $\cG_2^{(2a)}$ with all the internal faces of degree 4; the bipartiteness condition is equivalent to the fact that every boundary has even length, and implies that the $f_s$-\bgirth\ is 2 (hence $\Qbb{2a}\subset \cA_2^{(2a)}$).
Note that $\cAda$ is empty if $a<d$, and that $\cAdd$ identifies with the set of maps from $\cDd$ with a marked internal face of degree $d$.

For $M\in\cAda$ we define a \emph{blocked region} of $M$ as an internally-enclosed set of faces $S$ such that $f_s\notin S$ and $f_e\in S$.
Note that any blocked region has contour-length at least $d$, and that $S=\{f_e\}$ is a blocked region of contour-length $d$.
We also claim that if $S_1$ and $S_2$ are blocked regions of contour length $d$, then $S_1\cup S_2$ is a blocked region of contour length $d$. Indeed, for any blocked regions $S_1,S_2$ the sets $S_1\cup S_2$ and $S_1\cap S_2$ are clearly blocked regions. Moreover, denoting $\ell_1$, $\ell_2$, $\ell_{1\cup 2}$, and $\ell_{1\cap 2}$ the contour lengths of $S_1$, $S_2$, $S_1\cup S_2$ and $S_1\cap S_2$ respectively, we have 
$$\ell_{1}+ \ell_{2}=\ell_{1\cup 2}+ \ell_{1\cap 2}+2\Delta,$$
where $\Delta$ is the number of edges of $M$ incident to a face in $S_1\setminus S_2$ and a face in $S_2\setminus S_1$. 
Therefore if $\ell_1=\ell_2=d$, then
$$\ell_{1\cup 2}\leq \ell_{1}+ \ell_{2}-\ell_{1\cap 2}\leq d.$$
Thus there is a blocked region $S$ of contour length $d$ containing all the other blocked regions of contour length $d$. We call $S$ the \emph{maximal} blocked region of $M$.

A map $M\in\cAda$ is called \emph{reduced} if $S=\{f_e\}$ is the unique blocked region of contour-length $d$. We denote by $\cBda$ the subset of reduced maps from $\cAda$. We also denote by $\vcAda$ (resp. $\vcBda$) the set of maps from $\cAda$ (resp. from $\cBda$) with a marked corner of $f_s$. Note that $\cBda$ is empty for $a<d$ and that $\cBdd$ consists of exactly one map, with two faces of degree $d$, one internal and the other external. 

 
As in Sections~\ref{sec:general_case_quad_diss} and~\ref{sec:general_case_tri_diss}, we will consider a decomposition of maps in $\cAda$ into two parts, which, roughly speaking, are obtained by cutting the map along the contour of the maximal blocked region. 
Let $M$ be a map in $\cAda$.
Let $S$ be the maximal blocked region of $M$. 
Let $\Eo(S)$ be the set of edges of $M$  having both of their incident faces in $S$, and let $\Vo(S)$ be the set of vertices having all of their incident faces in $S$.   
It is easy to see that the region $H(S):=S\cup \Eo(S)\cup \Vo(S)$ is \emph{simply connected} (homeomorphic to a disk). We denote by $C(S)$ the cycle of $M$ corresponding to the boundary of the region $H(S)$. Hence the cycle $C(S)$, which is not necessarily simple, is made of the edges incident to both a face in $S$ and a face not in $S$. 
We denote by $M_1$ the map obtained from $M$ by keeping only the vertices in $\Vo(S)$, the edges in $\Eo(S)$, and the cycle $C(S)$ turned into a simple cycle $C'(S)$ (so a single vertex on $C(S)$ may correspond to several vertices of $C'(S)$). This process is illustrated in Figure~\ref{fig:decompose-maps}.
We denote by $f_s'$ the marked boundary face (of degree $d$) of $M_1$ which lies outside of the cycle $C'(S)$.      It is clear that $M_1$ is in $\cAdd$.  We denote by $M_2$ the map obtained from $M$ by erasing all the vertices in $\Vo(S)$ and all the edges in $\Eo(S)$, 
so that the simply connected region $H(S)$ is replaced by a single marked internal face (of degree $d$) 
which we denote by $f_e'$. Note that the contour $C(S)$ of the marked face $f_e'$ is not necessarily simple;   see Figure~\ref{fig:decompose-maps}.
It is clear that $M_2$ is in $\cBda$. 
Hence every map $M$ in $\cAda$ decomposes into a pair $(M_1,M_2)$ of maps in $\cAdd\times\cBda$. 

\fig{width=.8\linewidth}{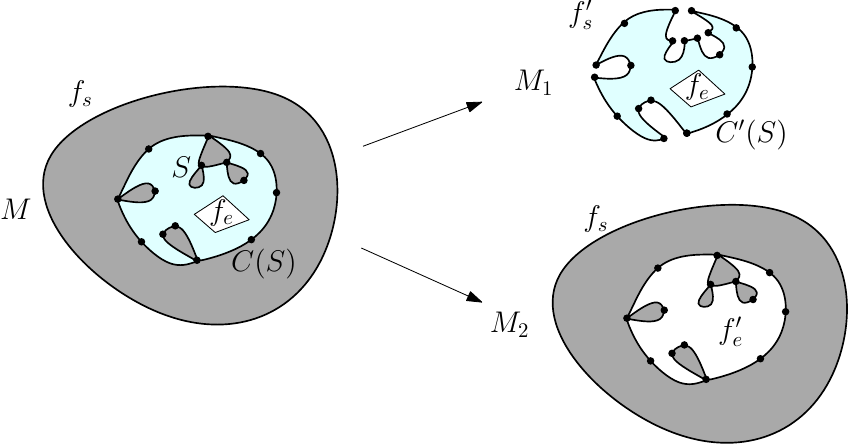}{The canonical decomposition of maps in $\cAda$ into a pair of maps in $\cAdd\times\cBda$.}

Conversely, given a pair $(M_1,M_2)$ in $\cAdd\times\cBda$, we can glue the contour of the marked boundary face $f_s'$ of $M_1$ to the contour of the marked internal face of $M_2$. Such a gluing produces a map in $\cAda$ such that the maximal blocked region of $M$ consists of all the faces of $M_1$ distinct from its marked boundary face $f_s'$. There are $d$ ways of gluing the two maps $M_1,M_2$ together, but one can easily make the gluing and ungluing canonical at the level of maps with marked corners (formally, one needs to fix an arbitrary convention for each map $M_2\in \vcBda$ by choosing one of the corners of the marked internal face as the ``gluing point'' for the marked corner of the maps in $\cAdd$). We call this the \emph{canonical decomposition} of the maps in $\cAda$. We summarize the above discussion:

\begin{lem}\label{lem:decompose-d}
The canonical decomposition is a bijection between $\vcAda$ and $\vcAdd\times\vcBda$.
\end{lem}


As mentioned earlier, the set of maps $\cAdd$ identifies with the set of maps from $\cDd$ with a (secondary) marked internal face of degree $d$. By Theorem~\ref{theo:bij_Dd}, the set $\cDd$ of maps is in bijection with the set of $\cTd$ of mobile via the master bijection. Moreover, through the master bijection, marking an internal face of degree $d$ in the map corresponds to marking a black vertex of degree $d$ of the mobile. We denote $\cUd$ the set of mobiles from $\cTd$ with a marked black vertex of degree $d$. We also denote $\cUdv$ the set of mobiles obtained from mobiles in $\cUd$ by marking one of the corners of the marked black vertex of degree $d$. The preceding remarks can be summarized as follows:
\begin{lem}\label{lem:marked-face-d}
For all $d\geq 1$, the set $\cAdd$ (resp. $\vcAdd$) of maps is in bijection with the set $\cUd$ (resp. $\cUdv$) of $\ZZ$-mobiles via the master bijection.
\end{lem}

\fig{width=\linewidth}{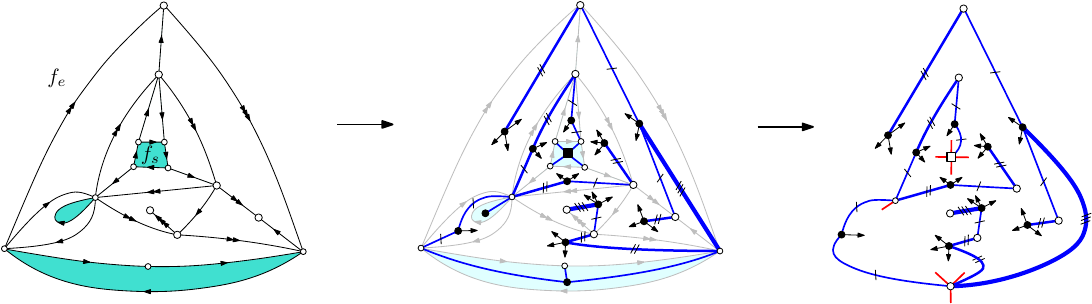}{Left: a map in $\cB_3^{(4)}$
endowed with its canonical orientation (all edges are 1-way
for $d=3$, the weight at the ingoing half-edge is indicated
as the multiplicity of the arrow). Right: the corresponding
mobile in $\cV_3^{(4)}$ (the marked white vertex is represented as a square; all edges are black-white for $d=3$, the weight at the 
white extremity is indicated as the number of stripes across the edge). 
}

We will now characterize maps in $\cBda$ by canonical orientations, in order to get a bijection for these maps as well. For a map $M\in\cGda$, we define a $d/(d-2)$-orientation of $M$ as
a consistent $\ZZ$-biorientation of $M$ with weights in 
$\{-2,-1,\ldots,d\}$ such that:
\begin{compactitem}
\item
every internal edge has weight $d-2$,
\item
every internal vertex has weight $d$, 
\item
every internal face $f$ has weight $d-\deg(f)$,
\item
every boundary face $f\neq f_s$ has weight $d+\deg(f)$, while  $f_s$ has weight~$-d+\deg(f_s)$. 
\end{compactitem}

\begin{prop}\label{prop:Adaori}
Let $M$ be a map in $\cGda$, let $f_s$ be its marked boundary face, and let $f_e$ be its marked internal face. We consider $M$ as a plane map by taking $f_e$ to be the outer face.
The map $M$ admits a $d/(d-2)$-orientation in $\wO_{d}$ if and only if $M$ is in $\cBda$ (that is, $M$ has $f_s$-\bgirth\ at least $d$, and is reduced).
In this case, the $d/(d-2)$-orientation in $\wO_{d}$ is unique. We call it the \emph{canonical biorientation} of~$M$. 

Lastly, in the case where $d$ and $a$ are both even, a map $M\in \cBda$ is bipartite if and only if all the internal half-edges of $M$ have even weight in its canonical biorientation. 
\end{prop}

The proof is delayed to Section~\ref{sec:proofs}. Note that
Proposition~\ref{prop:Adaori} includes Proposition~\ref{prop:Etri} (case $d=1$, with all internal faces of degree $3$), and also 
Proposition~\ref{prop2}
(case $d=2$, with all internal faces of degree $4$, and 
all boundary faces of even degree) upon dividing by $2$ all weights on internal half-edges.

For $a\geq 1$, we denote by $\cVda$ the set of mobiles corresponding to (canonically oriented) maps in $\cBda$ via the master bijection. By Theorem~\ref{theo:master_boundary}, these are the boundary $\ZZ$-mobiles with a marked white vertex satisfying the following properties (which readily imply that the weight of every half-edge is in $\{-2,-1,\ldots,d\}$ and the excess is 
$d$):
\begin{compactitem}
\item every edge has weight $d-2$, 
\item every black vertex has degree $\delta$ in $\{d,d+1,d+2\}$ and weight $-\delta+d$,
\item the marked white vertex has $a$ incident legs and has weight $a-d$, while every non-marked white vertex has weight $d+\ell$ where $\ell$ is the number of incident legs.
\end{compactitem}
We also denote by $\cVdav$ the set of rooted mobiles obtained from mobiles in $\cVda$ by marking one of the $a$ legs of the marked white vertex.

Proposition~\ref{prop:Adaori} together with the master bijection (Theorem~\ref{theo:master_boundary}) and Lemmas~\ref{lem:decompose-d} and~\ref{lem:marked-face-d} give the following result
(see Figure~\ref{fig:bij_d_annular} for an example).

\begin{theo}\label{theo:bij_Ada}
For all $d,a\geq 1$, the set $\cBda$ (resp. $\vcBda$) of maps is in bijection with the set $\cVda$ (resp. $\cVdav$) of $\ZZ$-mobiles via the master bijection.
Moreover, for all $d,a\geq 1$, the set $\vcAda$ of maps is in bijection with the set $\cUdv\times\cVdav$ of pairs of $\ZZ$-mobiles. The bijection is such that if the map $M$ corresponds to the pair of $\ZZ$-mobiles $(U,V)$, then every non-marked internal face of $M$ corresponds to a black vertex of $U\cup V$ of the same degree, every non-marked boundary face $f$ of $M$ corresponds to a non-marked white vertex of $U\cup V$ having $\mathrm{deg}(f)$ legs, and every inner vertex of $M$ corresponds to a white vertex of $U\cup V$ with no leg.
\end{theo}

\subsection{Expressions of the counting series} 
For $d\geq 1$ we denote by $\cF_d$ the set of maps from $\cDd$ with a marked corner in the marked boundary face.
We denote by $F_d\equiv F_d(x_d,x_{d+1},x_{d+2};z_0,z_1,\ldots)$ the associated series, with $x_\delta$ (for $\delta\in\{d,d+1,d+2\}$) conjugate to the number of internal faces of degree $\delta$, $z_0$ conjugate to the number of internal vertices, and $z_i$ (for $i\geq 1$) conjugate to the number of non-marked boundary faces of degree $i$. 

We will express $F_d$ as the difference of two series of mobiles. 
First, observe that marking a corner in maps from $\cDd$, is equivalent to marking an ingoing outer half-edge in canonically oriented maps from $\cDd$. Hence 
$$
F_d=K-L,
$$
where $K$ (resp. $L$) is the series for orientations from canonically oriented maps from $\cDd$ with a marked ingoing half-edge (resp. with a marked ingoing inner half-edge). 

Now we make the following observation about the master bijection $\hPhi$ (which easily follows from the definition of $\hPhi$):
\begin{claim}\label{claim:minus}
For $d\geq 1$, the master bijection $\hPhi$ yields bijections (with same parameter correspondences) between:
\begin{compactitem}
\item orientations from $\cOw_{-d}$ with a marked ingoing half-edge and mobiles from $\wB_{-d}$ with a marked bud,
\item orientations from $\cOw_{-d}$ with a marked ingoing inner half-edge and mobiles from $\wB_{-d}$ with a marked half-edge (possibly a leg) incident to a white vertex. 
\end{compactitem}
\end{claim}
By Theorem \ref{theo:bij_Dd} and Claim~\ref{claim:minus}, the series $K$ (resp. $L$) is also the series of mobiles from $\cT_d$ with a marked bud (resp. with a marked half-edge incident to a white vertex), with  $x_\delta$ conjugate to the number of black vertices of degree $\delta$, and $z_i$ conjugate to the number of white vertices with $i$ legs. 

Similarly as in Sections~\ref{sec:proof_counting_quad} and~\ref{sec:proof_counting_tri}, we introduce the notion of \emph{planted mobiles} to formulate recursive decompositions of mobiles. 
Precisely, for $d\geq 1$, 
a \emph{planted mobile} $\tau$ of type $d$ is defined as a tree structure that can be obtained as one of the two connected components after cutting a mobile $T\in\cT_d$ at the middle of an edge $e$ (where $e$ is a complete edge, not a bud nor a leg). 
The half-edge $h$ of $e$ belonging to the kept component $\tau$ is called the \emph{root} of $\tau$, the incident vertex is the \emph{root-vertex} of $\tau$ and the weight of $h$ is the \emph{root-weight} of $\tau$. 
For $j\in\{-2,\ldots,d\}$ we denote by $\cW_j$ the set of planted
mobiles of root-weight $d-2-j$, and we let $W_j=W_j(x_d,x_{d+1},x_{d+2};z_0,z_1,\ldots)$ be the associated
series, with as usual $x_\delta$ conjugate to the number of black vertices of degree
$\delta$ and $z_i$ conjugate to the number of 
white vertices with $i$ legs. 
Note that for $j\geq d-2$ the root-vertex is black while for $j<d-2$
the root-vertex is white. Similarly as in~\cite{BeFu12b}, we find for $j\in\{d-2,d-1,d\}$, 
$$
W_j=[u^{j+2}]\sum_{i=d}^{d+2}x_iu^i\big(1+W_0+u^{-1}W_{-1}+u^{-2}\big)^{i-1}.
$$
To treat the case of a white root-vertex we introduce the polynomials
$$
h_j^{(k)}(X_1,\ldots,X_d):=[u^j]\big(1-\sum_{i=1}^du^iX_i\big)^{-k}.
$$
For $j\in\{-2,\ldots,d-3\}$, the series $W_j$ is then given by
$$
W_j=\sum_{k\geq 0}z_k\ \!h_{k+2+j}^{(k+1)}(W_1,\ldots,W_d).
$$
Indeed, when the root-vertex $w_0$ carries $k$ legs, the total weight at $w_0$ is $k+d$, with a contribution $d-2-j$ from the root, and a total contribution $k+2+j$ from the other incident half-edges at $w_0$. 

Then we can express $K$ and $L$ in terms of the $W_j$'s. First, note that $K=W_{d-2}$ (indeed, a mobile from $\cT_d$ with a marked bud identifies to a mobile in $\cW_{d-2}$, upon seeing the marked bud
as a half-edge of weight $0$). We have
$$
L=\sum_{j=-2}^{d-3}W_jW_{d-2-j}+\sum_{k\geq 1}z_k\ \!h_{d+k}^{(k)}(W_1,\ldots,W_d),
$$
where the first sum gathers the cases where the marked half-edge
belongs to an edge, and the second sum gathers the cases 
where the marked half-edge is a leg. Finally, we obtain:

\begin{theo}\label{theo:Fd}
For $d\geq 1$ the counting series $F_d=F_d(x_d,x_{d+1},x_{d+2};z_0,z_1,\ldots)$ is given by
$$
F_d=W_{d-2}-\sum_{j=-2}^{d-3}W_jW_{d-2-j}-\sum_{k\geq 1}z_k\ \!h_{d+k}^{(k)}(W_1,\ldots,W_d),
$$
where the series $W_{-2},W_{-1},\ldots,W_d$ are determined by the system 
\begin{align*}
W_j&=\sum_{k\geq 0}z_k\ \!h_{k+2+j}^{(k+1)}(W_1,\ldots,W_d)\ \ \mathrm{for}\ j\in\{-2,\ldots,d-3\},\\[-.2cm]
W_j&=[u^{j+2}]\sum_{i=d}^{d+2}x_iu^i\big(1+W_0+u^{-1}W_{-1}+u^{-2}\big)^{i-1}\ \ \mathrm{for}\ j\in\{d-2,d-1,d\}.
\end{align*}
\end{theo}
Note that Theorem~\ref{theo:Fd} provides a generalization 
of the expression found in~\cite{BeFu12} for the series
of rooted $d$-angulations of girth $d$, for $d\geq 3$ (recovered here
by taking $x_{d+1}=x_{d+2}=0$ and $z_i=0$ for $i\geq 1$).

We can also derive expressions for the generating functions of maps in $\vcAda$. 
More precisely, we determine below the generating function $\vAda=\vAda(x_d,x_{d+1},x_{d+2};z_0,z_1,\ldots)$ of $\vcAda$, where the variable $x_{\delta}$ is conjugate to the number of non-marked  
internal faces of degree $\delta$, $z_0$ is conjugate to the number of 
internal vertices, and for $i\geq 1$ $z_i$ is conjugate to the number of non-marked
boundaries of length $i$. 

Let $\Ud\equiv \Ud(x_d,x_{d+1},x_{d+2};z_0,z_1,\ldots)$  be the series of mobiles from $\cUdv$, with 
 $x_\delta$ conjugate to the number of black vertices of degree $\delta$, and $z_i$ conjugate to the number of white vertices with $i$ legs. 
Since  the mobiles in $\cUdv$ are obtained from the mobiles in 
$\cT_d$ by marking a corner at a black vertex of degree $d$, we get 
$$
\Ud=(1+W_0)^d.
$$
For all $a\geq 1$, let $\Vd^{(a)}\equiv\Vd^{(a)}(x_d,x_{d+1},x_{d+2};z_0,z_1,\ldots)$ be the generating function of mobiles in $\cVdav$ with 
 $x_\delta$ conjugate to the number of black vertices of degree $\delta$, and $z_i$ conjugate to the number of  non-marked white vertices with $i$ legs. 
The decomposition at the marked vertex gives 
$$\Vd^{(a)}=h_{a-d}^{(a)}(W_1,\ldots,W_d).$$
Since Theorem \ref{theo:bij_Ada} yields $\vAda=\Ud\cdot\Vd^{(a)}$ we obtain the following result.

\begin{theo}\label{theo:Ada}
For $a,d\geq 1$ with $a\geq d$, the series $\vAda=\vAda(x_d,x_{d+1},x_{d+2};z_0,z_1,\ldots)$ is given by
$$
\vAda=(1+W_0)^d\cdot h_{a-d}^{(a)}(W_1,\ldots,W_d),
$$
with the same series $W_i$ as in Theorem~\ref{theo:Fd}. 
\end{theo}

Theorem~\ref{theo:Ada} provides a generalization 
of the expression found in~\cite{BeFu12} for the series
of rooted $d$-angulations of girth $d$ (for $d\geq 3$) with a boundary
of length $a$ (recovered here
by taking $x_{d+1}=x_{d+2}=0$ and $z_i=0$ for $i\geq 1$). 

For the case $d=2$, let $G^{(a)}(t,x_3,x_4)$ be the series $A_2^{(a)}$
under the specialization 
$\{x_2=0,z_0=t,z_i=0\ \mathrm{for}\ i\geq 1\}$. This gives
the enumeration of loopless maps with internal faces of degree
either $3$ or $4$ (with respective variables $x_3$ and $x_4$), 
a single boundary face of degree $a$, and a marked
edge (indeed the marked internal face of degree $2$ can be seen as collapsed into a marked edge). By Theorem~\ref{theo:Ada},  
$$
G^{(a)}(t,x_3,x_4)=(1+W_0)^2\cdot [u^{a-2}](1-uW_1-u^2W_2)^{-a},
$$
where the series $W_0,W_1,W_2$ are given by the system
\begin{align*}
W_{-2}&=t,\\
W_{-1}&=tW_1,\\
W_0\ \ &=2x_3W_{-1}(1+W_0)+3x_4W_{-2}(1+W_0)^2+3x_4W_{-1}\ \!\!^{2}(1+W_0),\\ W_1\ \ &=x_3(1+W_0)^2+3x_4W_{-1}(1+W_0)^2,\\
W_2\ \ &=x_4(1+W_0)^3. 
\end{align*}
 
Under the further specialization $\{x_3=1,x_4=0\}$, 
the system simplifies to $\{W_2=0, W_1=(1+W_0)^2, W_0=2t(1+W_0)^3\}$, and we find
\begin{align*}
G^{(a)}(t,1,0)&=(1+W_0)^2\cdot[u^{a-2}](1-uW_1)^{-a}\\
&=(1+W_0)^2\cdot\binom{2a-3}{a-2}W_1^{a-2}=\binom{2a-3}{a-2}(1+W_0)^{2a-2}. 
\end{align*}

The Lagrange inversion formula then gives
\begin{align*}
[t^n]G^{(a)}(t,1,0)&=\binom{2a-3}{a-2}\frac1{n}[y^{n-1}](2a-2)2^n(1+y)^{2a-3+3n}\\
&=2^{n+1}\frac{(2a-3)!}{(a-2)!^2}\cdot\frac{(2a-3+3n)!}{n!(2a-2+2n)!},
\end{align*}
The coefficient $[t^n]G^{(a)}(t,1,0)$ gives the number
of loopless triangulations with one boundary of length $a$, 
$n$ internal vertices, a marked corner in the boundary face, 
and a marked edge. Since such a triangulation has $2a+3n-3$ edges
(by the Euler relation), we recover Mullin's enumeration formula~\cite{mullin1965counting} whose first bijective proof 
was found in~\cite{PS03a}: 
\begin{coro}[\cite{mullin1965counting,PS03a}]\label{coro:mullin}
The number of loopless triangulations with one boundary of length $a$, 
$n$ internal vertices, and a marked corner in the boundary face is 
$\ds 2^{n+1}\frac{(2a-3)!}{(a-2)!^2}\cdot\frac{(2a-4+3n)!}{n!(2a-2+2n)!}$. 
\end{coro}

\vspace{.2cm}

Finally, similarly as in~\cite{BeFu12,BeFu12b}, the 
expressions simplify slightly in the bipartite case. 
Let $d=2b$ with $b\geq 1$. Let $\tF_{2b}:=F_{2b}$ under the
specialization $\{x_{2b+1}=0, z_{2i+1}=0\ \mathrm{for}\ i\geq 0\}$. 
And let $\tA_{2b}^{(2a)}:=A_{2b}^{(2a)}$ under the same 
specialization. Using the fact that, in the bipartite case, 
the weights of internal
half-edges are even in the canonical orientations (and thus, 
so are the half-edge weights in the associated mobiles), we easily
deduce from Theorems~\ref{theo:Fd} and~\ref{theo:Ada} the following expressions:

\begin{coro}[Bipartite case]\label{coro:bip}
The series $\tF_{2b}$ is expressed as
\begin{align*}
\tF_{2b}&=V_{b-1}-\sum_{j=-1}^{b-2}V_jV_{b-1-j}-\sum_{k\geq 1}z_{2k}h_{b+k}^{(2k)}(V_1,\ldots,V_b),
\end{align*}
where the series $V_{-1},V_{0},\ldots,V_b$ are determined by the system 
\begin{align*}
V_j&=\sum_{k\geq 0}z_{2k}\ \!h_{k+1+j}^{(2k+1)}(V_1,\ldots,V_b)\ \ \mathrm{for}\ j\in\{-1,\ldots,b-2\},\\[-.2cm]
V_{b-1}&=x_{2b}(1+V_0)^{2b-1}+(2b+1)\ \! x_{2b+2}V_{-1}(1+V_0)^{2b},\\[.2cm]
V_b &= x_{2b+2} (1+V_0)^{2b+1}.
\end{align*}
The series $\tA_{2b}^{(2a)}$ is expressed as 
$$
\tA_{2b}^{(2a)}=(1+V_0)^{2b}\cdot h_{a-b}^{(2a)}(V_1,\ldots,V_b).
$$
\end{coro} 

We now gives analogues of  Theorem~\ref{theo:krikun_triang} and Theorem~\ref{theo:krikun_quad} for other classes of triangulations and quadrangulations.
A map with boundaries having a marked boundary face $f_s$ is called \emph{internally loopless} (resp. \emph{internally simple}) if the $f_s$-internal girth is at least $2$ (resp. at least $3$).  
We give below factorized counting formulas (analogous to Krikun's formula \eqref{theo:krikun_triang})  when prescribing the number of internal vertices and the lengths of the boundaries  for internally loopless and internally simple triangulations, and for internally simple bipartite quadrangulations. 
This yields multivariate generalizations of the known formulas for the case of a single boundary, which were 
originally due to Mullin for loopless triangulations~\cite{mullin1965counting} 
(as already recovered in Corollary~\ref{coro:mullin}), 
and to Brown for simple triangulations~\cite{Br} and simple
quadrangulations~\cite{Br65}.

\begin{theo} \label{thm:krikun-analogues}
Let $m\geq 0$, $r\geq 0$, and let 
$a,a_1,\ldots,a_r$ be positive 
integers.  
For $s\in\{1,2,3\}$ let $\cT_{s}[m;a,a_1,\ldots,a_r]$ be the 
number of triangulations with 
 $m$ internal vertices,  
 $r+1$ (labeled from $0$ to $r$) 
boundary faces $f_0,f_1,\ldots,f_r$ such that $\mathrm{deg}(f_0)=a$, 
$\mathrm{deg}(f_i)=a_i$ for $1\leq i\leq r$, with a distinguished
corner in each boundary face, and such that
the internal girth with respect to $f_0$ is at least $s$. Note
that Theorem~\ref{theo:krikun_triang} gives a formula
for $|\cT_1[m;a,a_1,\ldots,a_r]|$. Analogously, we get  
$$
|\cT_2[m;a,a_1,\ldots,a_r]|=\frac{(2a-3)!}{(a-2)!^2}\cdot 2^{m+r+1}\cdot\frac{(3m+3r+2b-4)!}{m!(2m+2r+2b-2)!}\prod_{i=1}^ra_i\binom{2a_i+1}{a_i},
$$
$$
|\cT_3[m;a,a_1,\ldots,a_r]|=\frac{2(2a-3)!}{(a-3)!(a-1)!}\cdot\frac{(4m+4r+2b-5)!}{m!(3m+3r+2b-3)!}\prod_{i=1}^ra_i\binom{2a_i+2}{a_i},
$$
where  $b=a+a_1+\ldots+a_r$ is the total boundary length.

For $s\in\{1,2\}$ let  $\cQ_{s}[m;a,a_1,\ldots,a_r]$ be the 
number of quadrangulations with 
 $m$ internal vertices,  
 $r+1$ (labeled from $0$ to $r$) 
boundary faces $f_0,f_1,\ldots,f_r$ such that $\mathrm{deg}(f_0)=2a$, 
$\mathrm{deg}(f_i)=2a_i$ for $1\leq i\leq r$, with a distinguished
corner in each boundary face,  and such that
the internal girth with respect to $f_0$ is at least $2s$. Note
that Theorem~\ref{theo:krikun_quad} gives a formula
for $|\cQ_1[m;a,a_1,\ldots,a_r]|$.  Analogously, we get 
$$
|\cQ_2[m;a,a_1,\ldots,a_r]|=\frac{3(3a-2)!}{(a-2)!(2a-1)!}\cdot\frac{(3m+3r+3b-4)!}{m!(2m+2r+3b-2)!}\prod_{i=1}^r2a_i\binom{3a_i+1}{a_i},
$$ 
where  $b=a+a_1+\ldots+a_r$ is the half total boundary length.
\end{theo}
\begin{proof}
The proof is very similar to the proof of Theorem~\ref{theo:krikun_triang} and Theorem~\ref{theo:krikun_quad}. It simply relies on the Lagrange inversion formula starting from the generating function expressions given in Theorem~\ref{theo:Ada} and Corollary~\ref{coro:bip}.  
We treat first the case of internally loopless triangulations. We consider the series  $A^{(a)}_2=\binom{2a-3}{a-1}(1+W_0)^{2a-2}$, where $W_0$ is defined by   $W_0=2\sum_{k\geq 0}z_k\binom{2k+1}{k}(1+W_0)^{2k+3}$. Theorem~\ref{theo:Ada} gives 
$$
|\cT_2[m;a,a_1,\ldots,a_r]|=\frac1{e}\Big(\prod_{i\geq 1}i^{n_i}n_i!\Big)\cdot[z_0^{m}z_1^{n_1}\ldots z_h^{n_h}]A^{(a)}_2,
$$
where  $e=3m+3r+2b-3$ corresponds to the number of edges, $n_i$ is the number of occurrences of $i$ among $a_1,\ldots,a_r$, and $h$ is the maximum of $a_1,\ldots,a_r$.
Then, the Lagrange inversion formula gives\footnote{To apply the formula we let $n=m+n_1+\cdots+n_h$ and let $y=t\phi(y)$, with 
$\phi(y)=2\sum_{k\geq 0}z_k\binom{2k+1}{k}(1+y)^{2k+3}$. Then, with $\psi(y)=(1+y)^{2a-2}$, we have 
$[t^n]\psi(y)=\frac{1}{n}[y^{n-1}]\psi'(y)\phi(y)^n$, so that
$[z_0^{m}z_1^{n_1}\ldots z_h^{n_h}]\psi(y)=\frac{1}{n}[y^{n-1}]\psi'(y)[z_0^{m}z_1^{n_1}\ldots z_h^{n_h}]\phi(y)^n$.} 
$$
[z_0^{m}z_1^{n_1}\ldots z_h^{n_h}]A^{(a)}_2=\frac{(2a-2)!}{(a-2)!(a-1)!}2^{m+r}\cdot\frac{(3m+3r+2b-3)!}{m!(2m+2r+2b-2)!}\cdot\prod_{i\geq 1}\frac1{n_i!}\binom{2i+1}{i}^{n_i},
$$
which yields the formula for $|\cT_2[m;a,a_1,\ldots,a_r]|$. 

To treat the case of internally simple triangulations, we consider the series $A^{(a)}_3=\binom{2a-4}{a-1}(1+W_0)^{2a-3}$, where $W_0$ is given by $W_0=\sum_{k\geq 0}z_k\binom{2k+2}{k}(1+W_0)^{2k+4}$. Theorem~\ref{theo:Ada} gives 
$$
 |\cT_3[m;a,a_1,\ldots,a_r]|=\frac1{f}\Big(\prod_{i\geq 1}i^{n_i}n_i!\Big)\cdot[z_0^{m}z_1^{n_1}\ldots z_h^{n_h}]A^{(a)}_3,
$$
where $f=2m+2r+b-2$ corresponds to the number of internal faces,  $n_i$ is the number of occurrences of $i$ among $a_1,\ldots,a_r$, and $h$ is the maximum of $a_1,\ldots,a_r$.
Then, the Lagrange inversion formula gives 
$$
[z_0^{m}z_1^{n_1}\ldots z_h^{n_h}]A^{(a)}_3=\frac{(2a-3)!}{(a-3)!(a-1)!}\cdot\frac{(4m+4r+2b-4)!}{m!(3m+3r+2b-3)!}\cdot\prod_{i\geq 1}\frac1{n_i!}\binom{2i+2}{i}^{n_i},
$$
which yields the formula for $|\cT_3[m;a,a_1,\ldots,a_r]|$. 
 
Finally, for internally simple quadrangulations, we consider the series $\widetilde{A}^{(2a)}_4=\binom{3a-3}{a-2}(1+V_0)^{3a-2}$,  where $V_0$ is given by
 $V_0=\sum_{k\geq 0}z_{2k}\binom{3k+1}{k}(1+V_0)^{3k+3}$. Corollary~\ref{coro:bip}  gives 
$$
 |\cQ_2[m;a,a_1,\ldots,a_r]|=\frac1{f}\Big(\prod_{i\geq 1}(2i)^{n_i}n_i!\Big)\cdot[z_0^{m}z_2^{n_1}\ldots z_{2h}^{n_h}]\widetilde{A}^{(2a)}_4, 
$$
where $f=m+r+b-1$  corresponds to the number of internal faces,   $n_i$ is the number of occurrences of $i$ among $a_1,\ldots,a_r$, and $h$ is the maximum of $a_1,\ldots,a_r$.
Then, the Lagrange inversion formula gives 
$$
[z_0^{m}z_2^{n_1}\ldots z_{2h}^{n_h}]\widetilde{A}^{(2a)}_4=\frac{(3a-2)!}{(a-2)!(2a-1)!}\cdot\frac{(3m+3r+3b-3)!}{m!(2m+2r+3b-2)!}\cdot\prod_{i\geq 1}\frac1{n_i!}\binom{3i+1}{i}^{n_i},
$$
which yields the formula for $|\cQ_2[m;a,a_1,\ldots,a_r]|$. 
\end{proof}


\section{Existence and uniqueness of the canonical orientations}\label{sec:proofs}
In this section, we give the proof of Propositions~\ref{prop:Ddori} and~\ref{prop:Adaori} (which also imply Propositions \ref{prop:2outerquad}, \ref{prop2}, \ref{prop3}, and \ref{prop4}).

\subsection{Preliminary results}
We first set some notation and preliminary results about orientations. 
We call \emph{$\NN$-biorientation}, a $\ZZ$-biorientation with no negative weight (that is, the weight of every outgoing half-edge is 0, and the weight of every ingoing half-edge is a positive integer). 

\begin{Definition}\label{def:alphabeta}
Let $M$ be a map with boundaries. Let $V$ be the set of internal vertices, let $E$ be the set of internal edges, and let $B$ be the set of boundaries. Let $\alpha$ be a function from $V\cup B$ to $\NN$, and let $\beta$ be a function from $E$ to $\NN$.
An $\alpha/\beta$\emph{-orientation} of $M$ is a consistent $\NN$-biorientation, such that any vertex or boundary $x\in V\cup B$ has weight $\alpha(x)$, and any edge $e$ has weight $\beta(e)$.
\end{Definition}
The following two lemmas are immediate consequences of the results in~\cite[Lemma 2 and Lemma 3]{BeFu12} applied to the map obtained from $M$ by contracting each boundary into a single vertex.

\begin{lem}\label{lem:exists-alpha}
Let $M,V,E,B,\al,\be$ be as in Definition~\ref{def:alphabeta}. 
There exists a consistent $\alpha/\beta$-orientation of $M$ if and only if 
\begin{compactenum}
\item[(i)] $\sum_{x\in V\cup B}\alpha(x)=\sum_{e\in E}\beta(e)$,
\item[(ii)] for each subset $X\subseteq V\cup B$, $\sum_{x\in X}\alpha(x)\geq \sum_{e\in E_X}\beta(e)$, where $E_X\subseteq E$ is the subset of internal edges for which both endpoints are internal vertices in $X$ or boundary vertices incident to a boundary in $X$.
\end{compactenum}
Moreover, $\alpha/\beta$-orientations are accessible from an internal vertex $v$ (resp. boundary vertex $v$) if and only if
\begin{compactenum}
\item[(iii)] for each subset $X\subseteq V\cup B$ not containing $v$ (resp. not containing the boundary face incident to $v$), $\sum_{v\in X}\alpha(x)> \sum_{e\in E_X}\be(e)$.
\end{compactenum}
\end{lem}

\begin{lem}\label{lem:uniqueminimal}
Let $M$ be a plane map with boundaries, and let $V,E,B,\al,\be$ be as in Definition~\ref{def:alphabeta}. 
Suppose that there exists an $\alpha/\beta$-orientation $\Om$ of $M$. 
If the outer face of $M$ is an internal face, then $M$ admits a unique minimal $\alpha/\beta$-orientation $\Om_0$. 
If the outer face of $M$ is a boundary face $f$ satisfying $\alpha(f)=0$, then $M$ admits a unique almost-minimal $\alpha/\beta$-orientation $\Om_0$. 
In addition, in both cases, $\Om$ is accessible from a vertex $v$ if and only if $\Om_0$ is accessible from $v$. 
\end{lem}



Next, we state a parity lemma for orientations in $\cOw_d$.

\begin{lem}\label{lem:parity}
Let $O$ be a consistent $\ZZ$-biorientation in $\cOw_d$ 
(for some $d\in\mathbb{Z}\backslash\{0\}$), such that every internal edge, internal vertex, internal face, boundary, has even weight. Then every 
internal half-edge also has even weight. 
\end{lem}
\begin{proof}
Let $T$ be the boundary mobile associated with $O$ by the master bijection 
(Theorem~\ref{theo:master_boundary}). The parity conditions of $O$ imply that all edges and vertices of $T$
have even weight. In particular an edge $e$ of $T$ either has its two half-edges of odd weight, in which case $e$ is called \emph{odd}, or has its two half-edges of even weight, in which case $e$ is called \emph{even}. 
Let $F$ be the subforest of $T$ formed by the odd edges. Since every vertex of $T$ has even weight, it is incident to an even number of edges in $F$. Hence $F$ has no leaf, so that $F$ has no edge. Thus all edges of $T$ are even, and by the local rules of the master bijection it implies that all internal half-edges of $O$ have even weight. 
\end{proof}

\subsection{Regular orientations}
We now prove the existence of certain canonical orientations for bipartite maps of internal girth $2b$. This extends results proved in \cite{BeFu12} to maps with boundaries.

Recall that a \emph{$d$-angulation with boundaries} is a map such that every internal face has degree $d$.
Let $M$ be a bipartite $2b$-angulation with boundaries, and let $f_s$ be a distinguished boundary face.
We call \emph{$b/(b-1)$-orientation of $(M,f_s)$} an $\al/\be$-orientation of $M$ where,
\begin{compactitem}
\item $\al(v)=b$ for every internal vertex, and $\be(e)=b-1$ for every internal edge,
\item $\al(f)=\deg(f)/2+b$ for every boundary face $f\neq f_s$, and $\al(f_s)=\deg(f_s)/2-b$.
\end{compactitem}
 
We say that a vertex $x$ of $M$ is \emph{$d$-blocked from the face $f_s$} if there is an internally-surrounded set of faces $S$ containing all the faces incident to $x$ but not $f_s$, and having contour-length $d$. 

\begin{lem} \label{lem:b-orient}
Let $M$ be a bipartite $2b$-angulation with boundaries, and let $f_s$ be a distinguished boundary face. 
If $M$ has $f_s$-internal girth at least $2b$, then there exists a $b/(b-1)$-orientation of $(M,f_s)$. 
Moreover, any such orientation is accessible from the vertices incident to $f_s$, and also from any vertex $v$ which is not $2b$-blocked from $f_s$.
\end{lem}

\begin{proof} We want to use Lemma~\ref{lem:exists-alpha}. We start by checking Condition (i).
Let $V$, $E$, and $F$ be the sets of internal vertices, edges and faces respectively. 
Let $V'$, $E'$ and $B$ be the sets of boundary vertices, edges and faces respectively. 
By definition, 
$$\ds \sum_{x\in V\cup B}\alpha(x)-\sum_{e\in E}\beta(e)=b|V|+b(|B|-2)+\frac{1}{2}\sum_{f\in B}\deg(f)-(b-1)|E|.$$
Moreover, the Euler formula gives $|V|+|F|+|B|=2+|E|$ (because $|V'|=|E'|$), while the incidence relation between faces and edges gives $\sum_{f\in B}\deg(f)=|E'|$ and $2b|F|=2|E|+|E'|$. Using these identities gives $\ds \sum_{x\in V\cup B}\alpha(x)-\sum_{e\in E}\beta(e)=0$ as wanted.
We now check Condition (ii). Let $X\subseteq V\cup B$, let $V_X=V\cap X$, and let $B_X=B\cap X$. Let $V_X'$ and $E_X'$ be the sets of vertices and edges incident to faces in $B_X$ (so that $E_X\cup E_X'$ are the edges of $M$ with both endpoints in $V_X\cup V_X'$). Note that it is sufficient to check Condition (ii) for the subsets $X$ such that the graph $G_X:=(V_X\cup V_X',E_X\cup E_X')$ is connected. Indeed, the quantity $\sum_{x\in X}\alpha(x)-\sum_{e\in E_X}\beta(e)$ is additive over the connected components of $G_X$. So we now assume that $G_X$ is connected and consider the corresponding submap $M_X$ of $M$. 
By definition, $\ds \sum_{x\in X}\alpha(x)-\sum_{e\in E_X}\beta(e)=b|V_X|+b|B_X|-2b\ONE_{f_s\in X}+\frac{|E'_X|}{2}-(b-1)|E_X|$, where $\ONE_{f_s\in X}$ is 1 if $f_s\in X$ and $0$ otherwise. 
Let $F_X$ be the non-boundary faces of $M_X$. The Euler formula reads $|V_X|+|F_X|+|B_X|=2+|E_X|$, so
$$\ds \sum_{x\in X}\alpha(x)-\sum_{e\in E_X}\beta(e)=2b\ONE_{f_s\notin X} +\frac{|E'_X|}{2}+|E_X|-b|F_X|.$$
Now, any face $f\in F_X$ corresponds to a set of faces $S\subseteq F\cup B$ of $M$ which is internally-enclosed. Since $M$ has $f_s$-internal girth at least $2b$, this implies that the faces in $F_X$ have degree at least $2b$, except possibly for the face $f\in F_X$ containing $f_s$ (in the case $f_s\notin X$). Thus, the incidence relation gives $2b|F_X|\leq 2|E_X|+|E_X'| -(2b-1)\ONE_{f_s\notin X}$. Thus, 
$$\ds \sum_{x\in X}\alpha(x)-\sum_{e\in E_X}\beta(e)\geq \ONE_{f_s\notin X}.$$
This proves (ii) so there exists a $b/(b-1)$-orientation of $(M,f_s)$. Moreover (iii) is also true for any vertex incident to $f_s$, so that any $b/(b-1)$-orientation is accessible from the vertices incident to $f_s$.
 
Lastly, we consider a vertex $v$ such that a $b/(b-1)$-orientation $\Om$ of $(M,f_s)$ is not accessible from $v$ and want to show that $v$ is $2b$-blocked from $f_s$. Note that there is no directed path from $v$ to $f_s$ in $\Om$ (since $\Om$ is accessible from the vertices incident to $f_s$). 
Let $U$ be the set of vertices of $M$ from which there is a directed path toward $f_s$, and let $M'$ be the submap of $M$ made of $U$ and the edges with both endpoints in $U$. The vertex $v$ lies strictly inside a face $f$ of $M'$, and we consider the set $S$ of faces of $M$ corresponding to $f$. This is clearly an internally-enclosed set of faces of $M$ (since boundary faces are directed cycles). Moreover any edge of $M$ strictly inside $f$, having one endpoint on $f$, is oriented away from this vertex, so that the total weight $W_e$ of the edges strictly inside $f$ is equal to the total weight $W_v$ of the vertices strictly inside $f$. By combining the Euler relation and the incidence relation as above, the relation $W_e=W_v$ becomes $\deg(f)=2b$. 
Thus, the internally-enclosed set $S$ has contour-length $2b$. Hence, $v$ is $2b$-blocked from $f_s$.
\end{proof}

For $M$ a map with boundaries, we call \emph{star map} of $M$, and denote $\siM$, the map obtained from $M$ by inserting a vertex $v_f$, called \emph{star-vertex}, in each internal face $f$, and joining $v_f$ by an edge to each corner of $f$. The star map $\siM$ is considered as a map with boundaries (same boundaries as $M$). A star map is shown in Figure~\ref{fig:star-map}. The vertices and edges of $\siM$ which are in $M$ are called $M$-vertices and $M$-edges, while the others are called star-vertices and star-edges.
If $M$ is bipartite and $f_s$ is a distinguished boundary, we call \emph{$b$-regular orientation of $(\siM,f_s)$} an $\al/\be$-orientation of $\siM$ where,
\begin{compactitem}
\item $\al(v)=b$ for every internal $M$-vertex, and  $\al(v)=\deg(v)/2+b$ for every star-vertex,
\item $\al(f)=\deg(f)/2+b$ for every boundary face $f\neq f_s$, and $\al(f_s)=\deg(f_s)/2-b$,
\item $\be(e)=b-1$ for every internal $M$-edge, and $\be(e)=1$ for every star-edge.
\end{compactitem}

\fig{width=.8\linewidth}{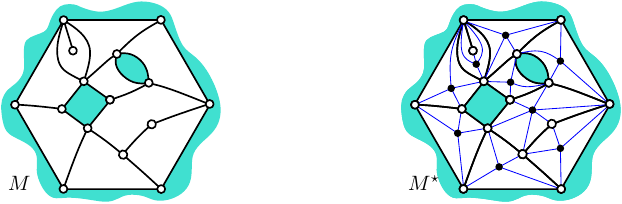}{A map with boundaries $M$, and the corresponding star map $\siM$ (with star-vertices colored black and $M$-vertices colored white).}

\begin{prop} \label{lem:b-regular-orient}
Let $M$ be a bipartite map with boundaries, and let $f_s$ be a distinguished boundary face. 
If $M$ has $f_s$-internal girth at least $2b$, then there exists a $b/(b-1)$-regular orientation of $(\siM,f_s)$. Moreover, any such orientation is accessible from the vertices incident to $f_s$, and also from any $M$-vertex $v$ which is not $2b$-blocked from $f_s$.
\end{prop}

\fig{width=.8\linewidth}{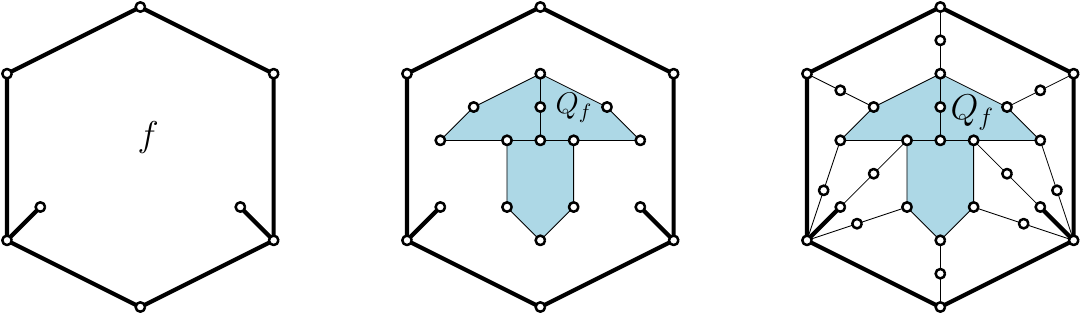}{Filling a face of $M$ with a $2b$-angulation without reducing the $f_s$-internal girth (here $b=3$): first insert a $2b$-angulation $Q_f$ with one boundary $f'$ of length $\deg(f)$ inside $f$ (middle picture), and then connect each corner of $f$ to a distinct corner of $f'$ by a path of length $b-1$ (right picture). 
}

\begin{proof}
We start by defining an orientation of $\siM$, and we will show later that it has the desired properties.
We first construct a $2b$-angulation with boundaries $N$, by inserting a $2b$-angulation in each internal face of $M$ in the manner illustrated in Figure~\ref{fig:angulate_f}. 
Precisely, for each internal face $f\neq f_s$ of $M$, we do the following: 
\begin{compactitem}
\item We insert a $2b$-angulation $Q_f$ with one boundary $f'$ of length $\deg(f)$. The $2b$-angulation $Q_f$ is chosen to have girth $2b$,  and $Q_f$ is placed so that $f$ and $f'$ are facing each other as in Figure~\ref{fig:angulate_f}.
\item We join by an edge each corner $c$ of $f$ to a distinct corner of $f'$ by a path $P_c$ of length $b-1$ (without creating edge crossings).
\end{compactitem}
It is easy to see that $N$ has $f_s$-internal girth $2b$.  
Moreover, we can choose the maps $Q_f$ to be chordless (i.e., no 
inner edge of $Q_f$ joins two vertices on the outer contour
of $Q_f$). This easily ensures that any vertex of $M$ which is not $2b$-blocked from $f_s$ in $M$ is not $2b$-blocked from $f_s$ in $N$ (because no cycle of length at most $2b$ using some edges of $N\setminus M$ can surround a vertex of $M$). 
Since $N$ has $f_s$-internal girth $2b$, Lemma~\ref{lem:b-regular-orient} ensures the existence of a $b/(b-1)$-orientation $X$ of $N$. It is easy to see that, for each corner $c$ of $M$, the weight $w_c$ of the half-edge of the path $P_c$ incident to $c$ is either 0 or 1 (otherwise the weight of the $b-2$ vertices on the path $P_c$ cannot all be equal to $b$). We then define an orientation $Y$ of $\siM$ by replacing each quadrangulation $Q_f$ of $N$ by a star-vertex $v_f$, and replacing each path $P_c$ by a single edge $e_c$ from the corner $c$ to $v_f$ with weight $w_c$ on the half-edge incident to $c$ and $1-w_c$ on the half-edge incident to $v_f$.
We now prove that $Y$ is a $b$-regular orientation of $(\siM,f_s)$.
First, it is clear that the weight of every $M$-edge is $b-1$, and the weight of every star-edge is 1. Second it is clear that the weight of every $M$-vertex is $b$, and for every boundary face $f$ the
weight of the corresponding boundary    
is $\deg(f)/2+b$ (same as in $X$). Thus it only remains to check that the weight of each star-vertex $v_f$ is $\deg(f)/2+b$. Let $W_f$ be the total weight, in $X$, of the half-edges incident to $Q_f$ on the paths $P_c$. It is easy to see that the weight of $v_f$ in $Y$ is $W_f$ (because for any corner $c$ the half-edge of $P_c$ incident to $Q_f$ has weight $1-w_c$), hence we need to show that $W_f=b+\deg(f)/2$.
Let $v$ and $e$ be the number of vertices and edges of $Q_f$. The Euler relation together with the face-edge incidence relation imply that $bv=(b-1)e+b+\deg(f)/2$. In $X$ the total weight of vertices in $Q_f$ is $bv$ and the total weight coming from edges in $Q_f$ is $(b-1)e$. Hence, $W_f=bv-(b-1)e=b+\deg(f)/2$, as wanted. Thus, $Y$ is a $b$-regular orientation of $(\siM,f_s)$.

We know from Lemma~\ref{lem:b-regular-orient} that the orientation $X$ of $N$ is accessible from the vertices incident to $f_s$, and from any vertex which is not $2b$-blocked from $f_s$. Since the accessibility properties can only improve when contracting $X$ into $Y$, the orientation $Y$ of $\siM$ is also accessible from the vertices incident to $f_s$, and from any 
vertex of $M$ which is not $2b$-blocked from $f_s$.
\end{proof}

\subsection{Transfer lemma}
Let $N$ be a map and $\siN$ be the corresponding star map.
We say that an $\NN$-orientation of $\siN$ is \emph{transferable} if for any star-edge $\eps$ with an ingoing half-edge incident to an $N$-vertex $v$, the $N$-edge $e$ following $\eps$ in clockwise order around $v$ is oriented 1-way toward $v$. This property is illustrated in Figure~\ref{fig:transfer}(a).


\fig{width=\linewidth}{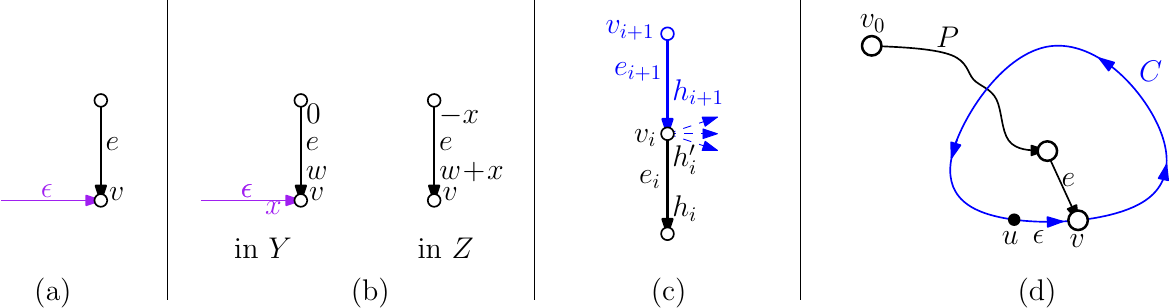}{(a) Condition for an orientation of $\siN$ to be \emph{transferable}: if the star-edge $\eps$ has an ingoing half-edge incident to a $N$-vertex $v$, then the following $N$-edge $e$ is oriented 1-way toward $v$. (b) Illustration of the transfer rule from the orientation $Y$ of $\siN$ to the orientation $Z$ of $N$. (c) Construction of the edges $e_1,e_2,e_3,\ldots$ forming a path directed toward $v$. (d) Proof of the uniqueness of the $\ZZ$-biorientation in $\wO_d$.}

\begin{lem}\label{lem:transfer}
Let $N$ be a plane map with boundaries having its outer face $f_0$ of degree $d$, and let $\siN$ be the corresponding star map.
Let $V$ be the set of internal $N$-vertices, let $E$ be the set of internal $N$-edges, and let $B$ be the set of boundaries of $\siN$. Let $V'$ be the set of star-vertices, and let $E'$ be the set of star-edges of $\siN$. Let $\alpha:V\cup V'\cup B\to \NN$, and let $\beta: E\cup E'\to \NN$, be functions such that 
there exists an $\al/\be$-orientation of $\siN$ which is accessible from the outer vertices of $N$. 

In the case where the outer face $f_0$ is a boundary face, we consider the star map $\siN$ as a plane map with outer face $f_0$. Suppose that $\al(f_0)=0$ and moreover the almost-minimal $\al/\be$-orientation of $\siN$ (which is unique by Lemma~\ref{lem:uniqueminimal}) is transferable. In this case, there exists a unique $\ZZ$-biorientation in $\wO_{-d}$ such that 
\begin{compactitem}
\item[(i)] 
any edge $e\in E$ has weight $\be(e)$, and any vertex $v\in V$ has weight $\al(v)$, 
\item[(ii)] any boundary face $f\in B$ has weight $\al(f)$, and any internal face $f$ has weight $\ds\al(v_f)-\sum_{e\in E'\textrm{ incident to }v_f}\be(e)$.
\end{compactitem}

In the case where the outer face $f_0$ is an internal face, we can consider $\siN$ as a plane map by choosing a face $\tilde{f}$ incident to the star-vertex $v_{f_0}$ to be the outer face of $\siN$. Suppose that $\ds \al(v_{f_0})=\sum_{e\in E'\textrm{ incident to }v_{f_0}}\be(e)$, and $\be(e)>0$ for every $N$-edge $e$ incident to $f_0$, and moreover the minimal $\al/\be$-orientation of $\siN$ (which is unique by Lemma~\ref{lem:uniqueminimal}) is transferable. In this case, there exists a unique $\ZZ$-biorientation of $N$ in $\wO_d$ satisfying the conditions (i-ii) above.
\end{lem}

\begin{proof}
We treat the case where $f_0$ is an internal face; the case where $f_0$ is a boundary face is almost identical. Let $Y$ be the unique minimal $\al/\be$-orientation of $\siN$. We are assuming that the $\NN$-orientation $Y$ is accessible from the outer vertices of $N$, that $\ds \al(v_{f_0})=\sum_{e\in E'\textrm{ incident to }v_{f_0}}\be(e)$, that $\be(e)>0$ for every $N$-edge $e$ incident to $f_0$, and that $Y$ is transferable. We need to prove the existence of a unique $\ZZ$-biorientation of $N$ in $\wO_d$ satisfying the conditions (i-ii).

We first define a $\ZZ$-biorientation $Z$ of $N$ starting from $Y$. The biorientation $Z$ is obtained by keeping the orientation of the $N$-edges as in $Y$, and putting weights according to the following \emph{transfer rule} for every internal $N$-edge $e$:
\begin{compactitem}
\item If $e$ is is 0-way or 2-ways, then the weights of the half-edge of $e$ are kept as in the orientation $Y$.
\item If $e$ is oriented 1-way, we consider the star-edge $\eps$ preceding $e$ clockwise around the end $v$ of $e$. Denoting $x$ the weight of the half-edge of $\eps$ incident to $v$, we set the weight of the outgoing half-edge of $e$ to $-x$, and we add $x$ to the weight of the ingoing half-edge of of $e$.
\end{compactitem}
The transfer rule is illustrated in Figure~\ref{fig:transfer}(b).
Because $Y$ is a transferable consistent $\al/\be$-orientation of $\siN$, it is easy to see that $Z$ is a consistent $\ZZ$-biorientation satisfying conditions (i-ii). 

We now want to prove that $Z$ is in $\wO_{d}$. Since $Z$ is consistent, this amounts to proving that $Z$ is in $\cO_{d}$.
First note that the minimality of $Y$ clearly implies the minimality of $Z$ (since any ccw-cycle of $Z$ would be a ccw-cycle of $Y$). Second, we prove that every $N$-edge $e$ incident to $f_0$ is either 2-way or 1-way with an inner face on its right in $Z$. Since $Z$ is a $\ZZ$-biorientation and $\beta(e)>0$ we already know that $e$ is either 2-way or 1-way. Suppose by contradiction that the edge $e=(u,v)$ is 1-way with $f_0$ on its right. By hypothesis, the orientation $Y$ is accessible from $v$, so there is a directed path $P$ from $v$ to $u$ in $Y$. The path $P$ does not pass through $v_{f_0}$ since the condition $\ds \al(v_{f_0})=\sum_{e\in E'\textrm{ incident to }v_{f_0}}\be(e)$ implies that the star-edges incident to $v_{f_0}$ are either 0-way or 1-way directed toward $v_{f_0}$. Thus, the cycle $P\cup\{e\}$ is a ccw-cycle of $Y$. This contradicts the minimality of $Y$. This completes the proof that every $N$-edge $e$ incident to $f_0$ is either 2-way or 1-way with an inner face on its right in $Z$. 
Lastly, we need to prove that $Z$ is accessible from every outer vertex. It suffices to prove that for every internal $N$-vertex $v$ there is a directed path $P_v$ of $Y$ made of $N$-edges going from an outer vertex to $v$. 
Let $v$ be an inner vertex. Since $v$ is accessible from the outer vertices in $Y$, there exists an ingoing half-edge incident to $v$. Moreover, since $Y$ is transferable, there exists a ingoing half-edge $h_1$ incident to $e$ and belonging to an $N$-edge $e_1$. Let $h_1'$ be the other half-edge of $e_1$, and let $v_1$ be the endpoint of $h_1'$. If $v_1$ is an outer vertex, then we can take the path $P_v=(e_1)$. Otherwise, we consider the first ingoing half-edge $h_2$ of $\siN$ following $h_1'$ counterclockwise around $v_1$. This construction is illustrated in Figure~\ref{fig:transfer}(c). Because $Y$ is transferable, $h_2$ is part of an $N$-edge $e_2$. Let $h_2'$ be the other half-edge of $e_2$, and let $v_2$ be the endpoint of $h_2'$. If $v_2$ is an outer vertex, then we can take the path $P_v=(e_2,e_1)$. Otherwise we continue and define a sequence $N$-edges $e_3,e_4,\ldots$ and the $N$-vertices $v_3,v_4,\ldots$ according to the same process. We claim that there exists $i>0$ such that $v_i$ is an outer vertex, so that one can take the path $P_v=(e_i,e_{i-1},\ldots,e_1)$. Indeed, suppose by contradiction that this is not true. In this case, there exists $i<j$, with $e_i= e_j$ and we consider the directed cycle $C=(e_i,e_{i+1},\ldots,e_{j-1})$. Without loss of generality we can assume that $C$ is a simple cycle, and since $Y$ is minimal it is a cw-cycle. Hence, all the outer-vertices are in the part of $N$ on the left of $C$. However, by construction, all the half-edges of $\siN$ incident to the vertices of $C$ and on the left of $C$ are outgoing. This shows that the vertices of $C$ are not accessible from the outer vertices in $Y$. This is a contradiction. This concludes the proof that $Z$ is accessible from every outer vertex, hence that $Z$ is in $\wO_d$.

It only remains to prove that there is no $\ZZ$-biorientation $Z'\neq Z$ in $\wO_{d}$ satisfying (i-ii). Suppose, by contradiction, that $Z'$ is such an orientation. It is easy to see that there exists a transferable consistent $\al/\be$-orientation $Y'\neq Y$ of $\siN$ such that $Z'$ is obtained from $Y'$ by the transfer rule defined above. Since $Y$ and $Y'$ are both $\al/\be$-orientations, Lemma~\ref{lem:uniqueminimal} ensures that $Y'$ is accessible from the outer vertices of $N$ but is not minimal. Since $Y'$ is not minimal, there exists a ccw-cycle $C$ of $Y'$ which encloses no other ccw-cycles. Since $Z'$ is minimal, and the orientation of $N$-edges in $Y'$ and $Z'$ coincide, there must be a star-vertex $u$ on $C$. The star-edge $\eps$ following the vertex $u$ on $C$ has an ingoing half-edge incident to an $N$-vertex $v$. And since $Y'$ is transferable, the $N$-edge $e$ preceding $\eps$ around $v$ is oriented 1-way toward $v$ in $Y'$. Note that $e$ is enclosed by $C$; see Figure~\ref{fig:transfer}(d). Now consider a directed path $P$ in $Y'$ going from an outer vertex of $N$ to the origin of $e$. It is clear that $C\cup P\cup \{e\}$ contains a ccw-cycle enclosed in $C$. This contradicts our choice of $C$, and completes the proof.
\end{proof}

\subsection{Proof of Proposition~\ref{prop:Ddori}}
In this subsection we prove Proposition~\ref{prop:Ddori}. In this entire subsection, $M$ is a map in $\cEd$, and $f_s$ is its marked boundary face. We start by showing the necessity of the internal girth condition.
\begin{claim}\label{claim:necessity-girth} 
If $M$ has a $d/(d-2)$-orientation, then it has $f_s$-\bgirth\ at least $d$.
\end{claim}

\begin{proof}
Suppose that $M$ has a $d/(d-2)$-orientation $\Om$. Let $S$ be a non-empty internally-enclosed set of faces not containing $f_s$. We want to show that the contour length $\ell$ of $S$ is at least $d$. We can therefore assume, without loss of generality, that the set $S$ of faces forms a connected region. Let $V$ and $E$ be respectively the sets of internal vertices and internal edges of $M$ incident only to faces in $S$.  Let $W$ be the total weight of edges in $E$ (for the $d/(d-2)$-orientation $\Om$ of $M$). By definition of $d/(d-2)$-orientations, 
$$(d-2)|E|=W\geq\sum_{v\in V}w(v)+\sum_{f\in S}w(f)= d|V|+d|S|-\sum_{f\in S,\textrm{ internal}}\deg(f)+ \sum_{f\in S,\textrm{ boundary}}\deg(f).$$
Moreover the incidence relation between edges and faces gives
$$\sum_{f\in S,\textrm{ internal}}\deg(f)=2|E|+ \ell+\sum_{f\in S,\textrm{ boundary}}\deg(f),$$
and the Euler relation gives $|V|-|E|+|S|=1$. Combining these relations gives $\ell\geq d$ as wanted.
\end{proof}

From now on, we assume that $M$ is in $\cDd$. We need to show the existence of a unique $d/(d-2)$-orientation of $M$. Let $N$ be the map with boundaries obtained from $M$ by inserting 3 vertices, called \emph{edge-vertices}, on each edge of $M$. For every face $f$ of $M$, we denote $f'$ the corresponding face of $N$. For every boundary $f$ of $M$, we consider $f'$ as a boundary of $N$. We also consider $f_s'$ as the outer face of $N$.
 Observe that $N$ is a bipartite map of $f_s'$-internal girth $4d$. Hence, by Proposition~\ref{lem:b-regular-orient}, there exists a $2d/(2d-1)$-regular orientation of  $(\siN,f_s')$. Moreover, any such orientation is accessible from the vertices incident to $f_s'$. By Lemma~\ref{lem:uniqueminimal} there exists a unique almost-minimal $2d/(2d-1)$-regular orientation $X$ of $(\siN,f_s')$. 

\begin{claim} \label{claim:transferable}
The orientation $X$ of $\siN$ is transferable. 
\end{claim}

\begin{proof}
We consider an edge $e$ of $\siN$ oriented from a star-vertex $u$ to an $N$-vertex $v$, and consider the $N$-edge $e'=\{w,v\}$ following $e$ in clockwise direction around $v$. We want to show that $e'$ is oriented 1-way toward $v$. In Figures~\ref{fig:transferable-gle1} and~\ref{fig:transferable-gle2}, we suppose by contradiction that the half-edge of $e'$ incident to $w$ is ingoing. 

Let us suppose first that $v$ is a vertex of $M$, as in Figure~\ref{fig:transferable-gle1}. By definition, the indegree of $u$ is $2d+\deg(u)/2$, and $\deg(u)\in\{4d,4(d+1),4(d+2)\}$, so the outdegree of $u$ is at most 4. Now, observe that the star-edge $\{u,w\}$ must be oriented toward $w$ to avoid creating a ccw-cycle. Since $w$ is either a boundary vertex or a vertex of weight $2d$, the $N$-edge $e''\neq e'$ incident to $w$ is not oriented 1-way toward $w$. Iterating this reasoning three times shows that $u$ has outdegree at least 5, which is a contradiction. 

\fig{width=.9\linewidth}{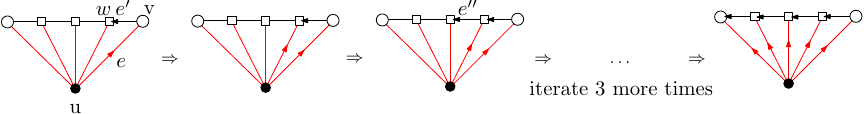}{Proof that $M$ is transferable: case where $v$ is a vertex of $M$.}
\fig{width=\linewidth}{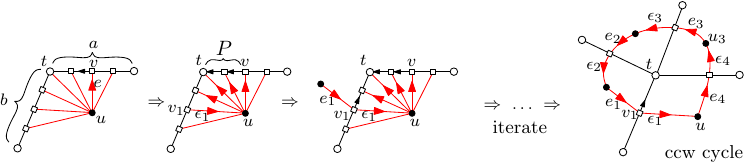}{Proof that $M$ is transferable: case where $v$ is an edge-vertex of $N$.}

We now suppose that $v$ is an edge-vertex of $N$, as in Figure~\ref{fig:transferable-gle2}. Let $f$ be the face of $N$ containing $u$, let $a$ be the edge of $M$ on which $v$ lies, let $b$ be the edge of $M$ preceding $a$ clockwise around $f$, and let $t$ be the vertex of $M$ preceding $v$ clockwise around $f$. 
Let $P$ be the path of $N$ made of the edges between $v$ and $t$ on $a$. Reasoning as above we conclude that $P$ is a directed path from $v$ to $t$ (its edges are either 2-ways or 1-way in the direction of $P$) and the star-edges between $P$ and $u$ are oriented toward $P$.
We now consider the first star-edge $\eps_1$ oriented toward $u$ following $e$ in counterclockwise order around $u$, and we denote $v_1$ its origin. 
Since $v_1$ is not on $P$ and the star-vertex $u$ has outdegree at most 4, $v_1$ must be an edge-vertex on $b$; see Figure~\ref{fig:transferable-gle2}. Moreover, the $N$-edge preceding $\eps_1$ in clockwise order around $v_1$ must be oriented 1-way away from $v_1$ (to avoid a ccw-cycle). This implies that $v_1$ is internal, and that the other star-edge $e_1$ incident to $v_1$ must be oriented toward $v_1$ (because the weight of $v_1$ is $2d$). At this point we can apply to $e_1$ the reasoning we just applied to $e$. Iterating this process proves the existence of some edges $\eps_1,e_1,\ldots,\eps_{\deg(t)},e_{\deg(t)}$, forming a ccw-cycle around $t$; see Figure~\ref{fig:transferable-gle2}. We again reach a contradiction, which completes the proof.
\end{proof}

We can now conclude the proof of Proposition~\ref{prop:Ddori}. We apply the transfer Lemma~\ref{lem:transfer} to $X$ (with the outer face $f_0$ being the marked boundary face $f_s'$ of $N$). It implies the existence of a unique $\ZZ$-orientation $Y$ of $N$ which is in $\wO_{-4d}$ and such that 
\begin{compactitem}
\item[(i)] any internal edge has weight $2d-1$, and any internal vertex $v$ has weight $2d$, 
\item[(ii)] $f_s'$ has weight 0, any boundary face $f'\neq f_s'$ of $N$ has weight $2d+\deg(f')/2=2d+2\deg(f)$,  
\item[(iii)] any internal face $f'$ of $N$  has weight $2d-\deg(f')/2=2d-2\deg(f)$.
\end{compactitem}
We now use the rules indicated in Figure~\ref{fig:fromYtoZ} in order to obtain from $Y$, an orientation $Z$ of $M$. It is easy to check that the orientation $Z$ of $M$ is in $\wO_{-d}$ (indeed the almost-minimality and accessibility are clearly preserved from $Y$ to $Z$) and such that 
\begin{compactitem}
\item[(i')] any internal edge has weight $2d-4$, and any internal vertex $v$ has weight $2d$, 
\item[(ii')] $f_s$ has weight 0, any boundary face $f\neq f_s$ of $M$ has weight $2d+2\deg(f)$,
\item[(iii')] any internal face $f$ of $M$ has weight $2d-2\deg(f)$.
\end{compactitem}
Moreover, by Lemma~\ref{lem:parity}, the weight of internal half-edges of the orientation $Z$ are all even. Upon dividing by 2 the weight of every internal half-edge, one obtains a $d/(d-2)$-orientation $\Om$ of $M$ in $\wO_{-d}$.

\fig{width=.8\linewidth}{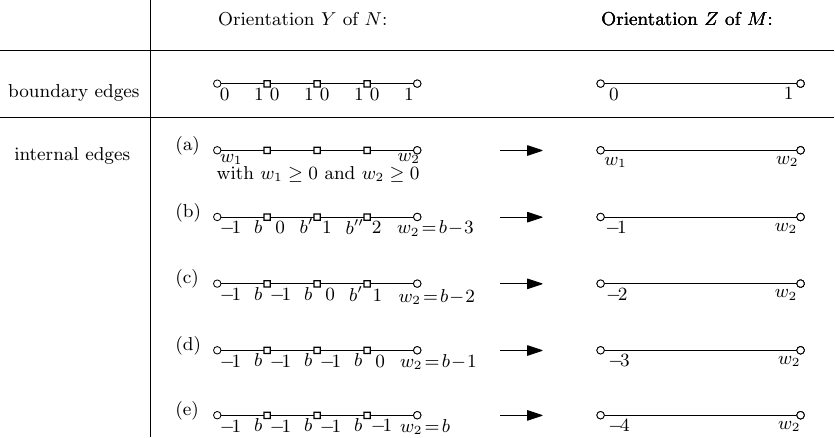}{Possible weights at a subdivided edge of $M$ for the $2d/(2d-1)$-orientation $Y$, and the associated orientation $Z$ of $M$. We use the notation $b=2d$, $b'=2d-1$, and $b''=2d-2$. Case (a-e) correspond to internal edges. In (a) the weights $w_1,w_2$ of the half-edges incident to the vertices of $M$ are both non-negative (in this case, $w_1+w_2=4(b-1)-3b=2d-4$ and the other weights are determined uniquely from $w_1,w_2$). In (b-e) we assume $w_1=-1$ and consider the possible sequences of weights.}

We now argue that $\Om$ is the unique $d/(d-2)$-orientation of $M$ in $\wO_{-d}$. Indeed, suppose by contradiction that there is another $d/(d-2)$-orientation $\Om'$ in $\wO_{-d}$. By doubling the weight of every internal half-edge, one obtains an orientation $Z'\neq Z$ satisfying the conditions (i'), (ii'), (iii') above, and upon inverting the rule represented in Figure~\ref{fig:fromYtoZ}, one gets an orientation $Y'\neq Y$ of $N$ in $\wO_{-4d}$ satisfying conditions (i), (ii), (iii) above. By inverting the ``transfer rule'', one would get from it an almost-minimal $2d/(2d-1)$-regular orientation $X'\neq X$ of $(\siN,f_s')$. This contradicts the uniqueness of the  almost-minimal $2d/(2d-1)$-regular orientation of $(\siN,f_s')$. This concludes the proof of the existence and uniqueness of a $d/(d-2)$-orientation $\Om$ of $M$ in $\wO_{-d}$. 

The additional statement about the parity of the weights in the case where $d$ is even is a direct consequence of Lemma~\ref{lem:parity}, hence the proof of Proposition~\ref{prop:Ddori} is complete.

\subsection{Proof of Proposition~\ref{prop:Adaori}}
In this entire subsection, $M$ is a map in $\cGda$, $f_s$ is its marked boundary face, and $f_e$ is its marked internal face $f_e$. We consider $M$ as a plane map by taking $f_e$ to be the outer face of $M$.

First, we observe that the following claim about the \bgirth\ condition has the exact same proof as Claim~\ref{claim:necessity-girth}.
\begin{claim}
If $M$ has a $d/(d-2)$-orientation, then it has $f_s$-\bgirth\ at least $d$.
\end{claim}
From now on we assume that $M$ has $f_s$-\bgirth\ at least $d$ (i.e. is in $\cAda$).
Next we prove the necessity of $M$ being reduced.  
\begin{claim}\label{claim:necessity-reduced} 
If $M$ has a $d/(d-2)$-orientation in $\wO_{d}$, then $M$ is reduced.
\end{claim}
\begin{proof}
Let $M$ be a map in $\cAda$ admitting a $d/(d-2)$-orientation $\Om$ in $\wO_{d}$. We denote by $w(x)$ the weight of a vertex or boundary $x$ in $\Om$. Let $S$ be a blocked region of contour length $d$. We need to show $S=\{f_e\}$. Let $V$ and $E$ be respectively the set of inner vertices, and inner edges of $M$ incident only to faces in $S$. Let $U$ be the set of vertices which are incident to both a face in $S$ and a face not in $S$, and let $H$ is the set of half-edges which are part of an edge in $E$ and whose incident vertex is in $U$. By definition of $d/(d-2)$-orientations, 
$$(d-2)|E|\geq \sum_{h\in H,w(h)> 0} w(h)+\sum_{v\in V}w(v)+\sum_{f\in S}w(f) = d|V|+d|F|-\sum_{f\in S,\textrm{ internal}}\deg(f)+ \sum_{f\in S,\textrm{ boundary}}\deg(f).$$
Combining this relation with the incidence relation between faces and edges ($\ds \sum_{f\in S,\textrm{ internal}}\deg(f)=2|E|+ d+\sum_{f\in S,\textrm{ boundary}}\deg(f)$) and the Euler relation ($|V|-|E|+|S|=1$) gives $\sum_{h\in H,w(h)> 0} w(h)\leq 0$. Hence, the weights of half-edges in $H$ are non-positive. Equivalently no edge of $E$ is oriented 1-way toward a vertex in $U$ or oriented 2-way with a vertex in $U$. Moreover, since $S$ is internally-enclosed, there cannot be any boundary edge incident only to faces of $S$ and oriented toward a vertex of $U$ either. This implies that no vertex of $V$ can reach a vertex of $U$. Since $\Om\in\wO_{d}$ is accessible from the vertices incident to $f_e$, it must be that one of the vertices incident to $f_e$ is in $U$. Moreover, since $\Om\in\wO_{d}$ all the edges incident to $f_e$ are 2-ways or 1-way with an inner face on its right. This implies that the edges on the contour of $S$ are incident to $f_e$. Thus $S=\{f_e\}$.
\end{proof}

From now on we assume that $M$ is reduced (i.e. is in $\cBda$). It remains to prove that $M$ has a unique $d/(d-2)$-orientation in $\wO_{d}$. The proof is similar to the proof of Proposition~\ref{prop:Ddori}. Let $N$ be the map with boundaries obtained from $M$ by inserting 3 vertices, called \emph{edge-vertices}, on each edge of $M$. For every face $f$ of $M$, we denote $f'$ the corresponding face of $N$. For every boundary $f$ of $M$, we consider $f'$ as a boundary of $N$. We also consider $f_e'$ as the outer face of $N$.
Observe that $N$ is a bipartite map of $f_s'$-internal girth at least $4d$.
Hence by Proposition~\ref{lem:b-regular-orient} there exists a $2d/(2d-1)$-regular orientation of the star map $(\siN,f_s')$. By Lemma~\ref{lem:uniqueminimal} there exists a unique minimal $2d/(2d-1)$-regular orientation $X$ of $(\siN,f_s')$.  Moreover, since $M$ is reduced, the vertices incident to $f_e'$ are not $4d$-blocked from $f_s$. Thus (by the second claim of Proposition~\ref{lem:b-regular-orient}), $X$ is accessible from the vertices incident to $f_e'$. We now want to apply the transfer Lemma~\ref{lem:transfer} to $X$ (with outer face the internal face $f_0=f_e'$). 
First, with the notation of Lemma~\ref{lem:transfer}, we have $\be(e)=1$ for every star-edge $e$ and $\al(v_{f_e'})=\deg(f_e')/2+2d=4d=\sum_{e\in E'\textrm{ incident to }v_{f_e'}}\be(e)$. Moreover, $X$ is shown to be transferable by the same proof as for Claim~\ref{claim:transferable}. Thus we can apply Lemma~\ref{lem:transfer}, which implies the existence of a unique $\ZZ$-orientation $Y$ of $N$ which is in $\wO_{4d}$ and such that 
\begin{compactitem}
\item[(i)] any internal edge has weight $2d-1$, and any internal vertex $v$ has weight $2d$, 
\item[(ii)] $f_s'$ has weight $2a-2d$, any boundary face $f'\neq f_s'$ of $N$ has weight $2d+\deg(f')/2$, 
\item[(iii)] any internal face $f'$ of $N$ has weight $2d-\deg(f')/2$.
\end{compactitem}
We now use the rules indicated in Figure~\ref{fig:fromYtoZ} in order to obtain from $Y$, an orientation $Z$ of $M$. It is easy to check that the orientation $Z$ of $M$ is in $\wO_{d}$ (indeed the minimality and accessibility are preserved from $Y$ to $Z$) and such that 
\begin{compactitem}
\item[(i')] any internal edge has weight $2d-4$, and any internal vertex $v$ has weight $2d$, 
\item[(ii')] $f_s$ has weight $2a-2d$, any boundary face $f\neq f_s$ of $M$ has weight $2d+2\deg(f)$, 
\item[(iii')] any internal face $f$ of $M$ has weight $2d-2\deg(f)$.
\end{compactitem}
Moreover, by Lemma~\ref{lem:parity}, the weight of internal half-edges of the orientation $Z$ are all even. Upon dividing by 2 the weight of every internal half-edge, one obtains a $d/(d-2)$-orientation $\Om$ of $M$ in $\wO_{d}$. 

The uniqueness of the $d/(d-2)$-orientation of $M$ in $\wO_{d}$ is proved in the same way as for Proposition~\ref{prop:Ddori}, and the additional statement about the parity of the weight in the case where $d,a$ are even follows from Lemma~\ref{lem:parity}. This completes the proof of Proposition~\ref{prop:Adaori}.

\section{Concluding remarks}\label{sec:extension}
In this article we gave a bijective proof to all the known enumerative results for maps with boundaries (such as Krikun's formula for triangulations with boundaries \eqref{eq:krikun_triang}), and also established new enumerative results (such as the formula for quadrangulations with boundaries \eqref{eq:krikun_quad}). 

However, in view of the results established in~\cite{BeFu12,BeFu12b,BeFu13},  one could have hoped to obtain enumerative results for slightly more general, or more natural, classes of maps with boundaries. Ideally one could hope to enumerate any class of maps with boundaries with control on the girth and on the degrees of both boundary and internal faces. Alas, the girth parameter we are able to control through our bijections, the \bgirth\, is a bit unnatural for maps with several boundaries (because it depends on a choice of a marked boundary). Second, for maps of \bgirth\ $d$ we could only allow internal faces degrees in $\{d,d+1,d+2\}$. These limitations were dictated by the proofs of our master bijection approach (existence of canonical orientations, and in particular Claim \ref{claim:transferable}), and we do not think that further improvements are possible in this direction. 

Looking at other approaches, we remark that one can count certain classes of maps with boundaries, with a control on a different girth parameter. Indeed, let us define the \emph{contractible girth} of a plane map with boundaries, as the smallest length of a simple cycle enclosing a region with no boundary face (note that \emph{girth} $\leq$ \emph{blocking girth} $\leq$ \emph{contractible girth}). It is not hard to see that using a generating function approach (by substitution) would allow one to compute the generating function of triangulations with boundaries of contractible girth $d=2$ or $3$ (starting from our results for $d=1$), and the generating function of bipartite quadrangulations with boundaries, of contractible girth $d=4$ (starting from our results for $d=2$). 
But we do not know if this method can be extended further, or if a direct bijective approach would work for these classes of maps. 
Still, there may be hope to obtain bijective results for other values of $d$ for maps having at most 2 boundaries (to be investigated\ldots).

\vspace{.2cm}

\ni{\bf Acknowledgments.} The authors thank J\'er\'emie Bouttier and Juanjo Ru\'e for very interesting discussions. Olivier Bernardi was partially supported by NSF grant DMS-1400859. \'Eric Fusy was partly supported by the ANR grant ``Cartaplus'' 12-JS02-001-01 and the ANR grant ``EGOS'' 12-JS02-002-01. 

\bibliographystyle{plain}
\bibliography{mabiblio}

\end{document}

%% file: local_def.tex
\usepackage{latexsym,amsmath, amsthm, amssymb,mathtools,dsfont}
\usepackage{geometry,paralist, color,graphicx}

\graphicspath{{./Figures/}}

\newtheorem{theo}{Theorem}[section]
\newtheorem{lem}[theo]{Lemma}

\newtheorem{claim}[theo]{Claim}
\newtheorem{prop}[theo]{Proposition}

\newtheorem{coro}[theo]{Corollary}
\newtheorem{Definition}[theo]{Definition}

\newcommand{\ds}{\displaystyle}
\newcommand{\fig}[3]{\begin{figure}[h!]\begin{center}\includegraphics[#1]{#2}\end{center}\caption{#3}\label{fig:#2}\end{figure}}
\DeclareMathOperator{\ONE}{\mathds{1}}

\newcommand{\al}{\alpha}
\newcommand{\be}{\beta}
\newcommand{\eps}{\epsilon}

\newcommand{\Om}{\Omega}
\newcommand{\siM}{M^{\star}}
\newcommand{\siN}{N^{\star}}

\newcommand{\NN}{\mathbb{N}}
\newcommand{\ZZ}{\mathbb{Z}}

\def\ni{\noindent}

\def\hPhi{\widehat{\Phi}}
\def\Cat{\textrm{Cat}}

\def\Eo{E^\circ}
\def\Vo{V^\circ}

\def\cA{\mathcal{A}}
\def\cB{\mathcal{B}}
\def\cC{\mathcal{B}}
\def\cD{\mathcal{D}}

\def\cE{\mathcal{E}}

\def\cG{\mathcal{G}}

\def\cQ{\mathcal{Q}}
\def\cO{\mathcal{O}}

\def\cF{\mathcal{F}}
\def\cT{\mathcal{T}}
\def\cU{\mathcal{U}}
\def\cV{\mathcal{V}}
\def\cW{\mathcal{W}}

\def\tF{\widetilde{F}}
\def\tA{\widetilde{A}}

\def\cEd{\cE_d}

\def\cDd{\cD_d}
\def\cTd{\cT_d}

\def\cAda{\cA_d^{(a)}}
\def\cGda{\cG_d^{(a)}}
\def\cAdd{\cA_d^{(d)}}
\def\cBdd{\cB_d^{(d)}}

\def\cUd{\cU_d}
\def\cUdv{\vec{\cU}_d}
\def\cVda{\cV_d^{(a)}}
\def\cVdav{\vec{\cV}_d^{(a)}}

\def\vcAdd{\vec{\cA}_d^{(d)}}
\def\vcAda{\vec{\cA}_d^{(a)}}
\def\vcBda{\vec{\cB}_d^{(a)}}

\def\vAda{{A}_d^{(a)}}

\def\cBda{\cB_d^{(a)}}

\def\bgirth{internal\ girth}

\def\cTqua{\cT_{\raisebox{-1pt}{\scriptsize$\Diamond$}}}

\def\cTtri{\cT_{\vartriangle}}
\def\cDqua{\cD_{\!\Diamond}}
\def\cDtri{\cD_{\vartriangle}}

\def\cUqua{\cU_{\raisebox{-1pt}{\scriptsize$\Diamond$}}}
\def\cUquav{\vec{\cU}_{\raisebox{-1pt}{\scriptsize$\Diamond$}}}
\def\cUtri{\cU_{\vartriangle}}

\def\cVqua{\cV_{\raisebox{-1pt}{\!\scriptsize$\Diamond$}}}
\def\cVquav{\vec{\cV}_{\raisebox{-1pt}{\!\scriptsize$\Diamond$}}}
\def\cVtri{\cV_{\vartriangle}}
\def\cVtriv{\vec{\cV}_{\vartriangle}}

\def\Uqua{U_{\raisebox{-1pt}{\!\scriptsize$\Diamond$}}}
\def\Utri{U_{\vartriangle}}
\def\Ud{U_{d}}
\def\Vqua{V_{\raisebox{-1pt}{\!\scriptsize$\Diamond$}}}
\def\Vtri{V_{\vartriangle}}
\def\Vd{V_{d}}

\def\vA{\vec{A}}
\def\vB{\vec{B}}

\newcommand{\vAqua}[1]{\vA_{\Diamond}^{\raisebox{-2pt}{\scriptsize$(#1)$}}}

\newcommand{\vAtri}[1]{\vA_{\vartriangle}^{\raisebox{-2pt}{\scriptsize$(#1)$}}}
\newcommand{\vBtri}[1]{\vB_{\vartriangle}^{\raisebox{-2pt}{\scriptsize$(#1)$}}}

\def\eps{\epsilon}

\def\hcQ{\widehat{\cQ}}
\def\hcT{\widehat{\cT}}

\def\cOw{\widehat{\cO}}

\def\cV{\mathcal{V}}

\def\wO{\widehat{\cO}}
\def\wB{\widehat{\cB}}
\def\wOd{\wO_d}
\def\wBd{\wB_d}

\newcommand{\Qu}[1]{\cD_{\!\Diamond}^{(#1)}}
\newcommand{\Qb}[1]{\overline{\cD}_{\!\Diamond }^{\raisebox{-2pt}{\scriptsize$(#1)$}}}

\newcommand{\Qbb}[1]{\cA_{\Diamond}^{(#1)}}
\newcommand{\Qvv}[1]{\vec{\cA}_{\Diamond}^{\raisebox{-2pt}{\scriptsize$(#1)$}}}
\newcommand{\Qbbr}[1]{\cB_{\Diamond}^{(#1)}}
\newcommand{\Qvvr}[1]{\vec{\cB}_{\Diamond}^{\raisebox{-2pt}{\scriptsize$(#1)$}}}

\newcommand{\Tr}[1]{\cD_{\!\vartriangle}^{(#1)}}
\newcommand{\Tb}[1]{\overline{\cD}_{\!\vartriangle}^{\raisebox{-2pt}{\scriptsize$(#1)$}}}
\newcommand{\Tbb}[1]{\cA_{\vartriangle}^{(#1)}}
\newcommand{\Tvv}[1]{\vec{\cA}_{\vartriangle}^{\raisebox{-2pt}{\scriptsize$(#1)$}}}
\newcommand{\Tbbr}[1]{\cB_{\!\vartriangle}^{(#1)}}
\newcommand{\Tvvr}[1]{\vec{\cB}_{\!\vartriangle}^{\raisebox{-2pt}{\scriptsize$(#1)$}}}